\documentclass[11pt,a4paper]{amsart}
\usepackage{amsmath, amsthm, amsfonts, amssymb}
\usepackage[bookmarks=false]{hyperref}

\usepackage{color} 

\numberwithin{equation}{section}

\newtheorem{letterthm}{Theorem}

\newtheorem{lettercor}[letterthm]{Corollary}

\newtheorem{thm}{Theorem}[section]
\newtheorem{lem}[thm]{Lemma}

\newtheorem{cor}[thm]{Corollary}
\newtheorem{prop}[thm]{Proposition}

\theoremstyle{definition}

\newtheorem*{claim}{Claim}

\newtheorem*{caseII_1}{Case type ${\rm II_1}$}
\newtheorem*{caseII_infty}{Case type ${\rm II_\infty}$}
\newtheorem*{caseIII}{Case type ${\rm III}$}

\newcommand{\R}{\mathbf{R}}
\newcommand{\C}{\mathbf{C}}
\newcommand{\Z}{\mathbf{Z}}
\newcommand{\F}{\mathbf{F}}

\newcommand{\N}{\mathbf{N}}

\newcommand{\Ad}{\operatorname{Ad}}
\newcommand{\id}{\text{\rm id}}

\newcommand{\Aut}{\operatorname{Aut}}
\newcommand{\rL}{\mathord{\text{\rm L}}}

\newcommand{\rE}{\mathord{\text{\rm E}}}

\newcommand{\tr}{\mathord{\text{\rm tr}}}
\newcommand{\core}{\mathord{\text{\rm c}}}

\newcommand{\Tr}{\mathord{\text{\rm Tr}}}

\newcommand{\Ball}{\mathord{\text{\rm Ball}}}

\newcommand{\ovt}{\mathbin{\overline{\otimes}}}

\newcommand{\AC}{\mathord{\text{\rm AC}}}
\newcommand{\Bic}{\mathord{\text{\rm B}}}

\newcommand{\dpr}{^{\prime\prime}}

\begin{document}

\title[Asymptotic structure of free product von Neumann algebras]{Asymptotic structure of free product \\ von Neumann algebras}

\begin{abstract}
Let $(M, \varphi) = (M_1, \varphi_1) \ast (M_2, \varphi_2)$ be the free product of any $\sigma$-finite von Neumann algebras endowed with any faithful normal states. We show that whenever $Q \subset M$ is a von Neumann subalgebra with separable predual such that both $Q$ and $Q \cap M_1$ are the ranges of faithful normal conditional expectations and such that both the intersection $Q \cap M_1$ and the central sequence algebra $Q' \cap M^\omega$ are diffuse (e.g.\ $Q$ is amenable), then $Q$ must sit inside $M_1$. This result generalizes the previous results of the first named author in \cite{Ho14} and moreover completely settles the questions of maximal amenability and maximal property Gamma of the inclusion $M_1 \subset M$ in arbitrary free product von Neumann algebras.
\end{abstract}

\author{Cyril Houdayer}
\address{Laboratoire de Math\'ematiques d'Orsay\\ Universit\'e Paris-Sud\\ CNRS\\ Universit\'e Paris-Saclay\\ 91405 Orsay\\ FRANCE}
\email{cyril.houdayer@math.u-psud.fr}
\thanks{CH is supported by ANR grant NEUMANN and ERC Starting Grant GAN 637601.}

\author{Yoshimichi Ueda}
\address{Graduate School of Mathematics \\ Kyushu University \\ Fukuoka, 810-8560 \\ JAPAN}
\email{ueda@math.kyushu-u.ac.jp}
\thanks{YU is supported by Grant-in-Aid for Scientific Research (C) 24540214.}


\maketitle

\section{Introduction and statement of the main results}

The first class of concrete maximal amenable subalgebras in von Neumann algebras was discovered by Popa in his seminal article \cite{Po83}. He showed that the generator maximal abelian subalgebra $\rL(\Z) = \rL(\langle a \rangle)$ is maximal amenable inside the free group factor $\rL(\F_2) = \rL(\langle a, b\rangle)$. Popa moreover introduced in \cite{Po83} a powerful method, based on the study of central sequences in the ultraproduct framework, to prove that a given amenable von Neumann subalgebra in a finite von Neumann algebra is maximal amenable. This method will be referred to as Popa's {\em asymptotic orthogonality property} in this paper. Popa's maximal amenability result \cite{Po83} for free group factors was recently extended by the first named author in \cite{Ho14} to a large class of free product von Neumann algebras, possibly of type ${\rm III}$. We refer to \cite{Ho14} and the references therein for further results on maximal amenability in the framework of von Neumann algebras. We point out that Boutonnet-Carderi recently introduced in \cite{BC14} a new method, based on the study of central states, to prove that a given amenable von Neumann subalgebra in a finite von Neumann algebra is maximal amenable. Among other things, they obtained concrete examples of maximal amenable von Neumann subalgebras in type ${\rm II_1}$ factors associated with higher rank lattices.

The aim of this paper is to further generalize the recent work of the first named author in \cite{Ho14} and to completely settle the questions of maximal amenability and maximal property Gamma of the inclusion $M_1 \subset M$ arising from an arbitrary free product $(M, \varphi) = (M_1, \varphi_1) \ast (M_2, \varphi_2)$.

We will say that an inclusion of von Neumann algebras $Q \subset M $ is with {\em expectation} if there exists a faithful normal conditional expectation from $M$ onto $Q$. Our first main result is the following optimal {\em Gamma stability} result inside arbitrary free product von Neumann algebras.

\begin{letterthm}\label{thmA}
For each $i \in \{1, 2\}$, let $(M_i, \varphi_i)$ be any $\sigma$-finite von Neumann algebra endowed with a faithful normal state. Denote by $(M, \varphi) = (M_1, \varphi_1) \ast (M_2, \varphi_2)$ the free product. Let $Q \subset M$ be any von Neumann subalgebra with separable predual and with expectation such that  $Q \cap M_1$ is diffuse and with expectation, and furthermore that $Q' \cap M^\omega$ is diffuse. Then we have $Q \subset M_1$.
\end{letterthm}

We refer to Theorem \ref{theorem-general} below for a more general statement that extend \cite[Theorem D]{Ho14} to arbitrary free product von Neumann algebras. As a corollary to Theorem \ref{thmA}, we infer that amenable von Neumann subalgebras $Q \subset M$ with expectation such that the intersection $Q \cap M_1$ is diffuse and with expectation must in fact sit inside $M_1$. Namely, we obtain the following result.

\begin{lettercor}\label{corB}
Let $(M, \varphi) = (M_1, \varphi_1) \ast (M_2, \varphi_2)$ be as in Theorem \ref{thmA}. Let $Q \subset M$ be any amenable von Neumann subalgebra with expectation such that  $Q \cap M_1$ is diffuse and with expectation. Then we have $Q \subset M_1$.
\end{lettercor}

We point out that the separability assumption on the predual of the amenable von Neumann subalgebra $Q \subset M$ is no longer needed in Corollary \ref{corB}. As we mentioned before, in the case when both $M_1$ and $M_2$ are {\em tracial} and both $\varphi_1$ and $\varphi_2$ are faithful normal {\em tracial} states, Corollary \ref{corB} is a consequence of \cite[Theorem D]{Ho14}. Very recently, Ozawa gave in \cite{Oz15} a short proof of the above Corollary~\ref{corB} in the {\em tracial} case using an idea in \cite{BC14}.  However, that proof depends upon the assumption that given states are tracial. Moreover, we would like to emphasize that the tools and the techniques we will develop in order to achieve the goal of this paper have strong potential in future research, and indeed lead to our next work on general rigidity phenomenon for free product von Neumann algebras \cite{HU15}.

We also point out that \cite[Theorem A and Corollary B]{Ho14} hold true under the more general assumption that $M_1$ is diffuse, instead of the centralizer $(M_1)^{\varphi_1}$ being diffuse as in \cite{Ho14}. In fact, we prove the optimal asymptotic orthogonality property result in arbitrary free product von Neumann algebras (see Theorem \ref{theorem-AOP} below) to make those assertions hold under such a general assumption. Remark that this generalization of \cite[Theorem A]{Ho14} does not follow from Theorem \ref{thmA}, since it is applicable to any intermediate subalgebra $M_1 \subset Q \subset M$ without {\it a priori} assuming it to be with expectation.

We now briefly explain the strategy of the proof of Theorem \ref{thmA}. To simplify the discussion, we will further assume that $Q \subset M$ is a {\em subfactor}. We refer to Section \ref{proofs-main-results} for further details.

{\bf Assume that $Q$ is amenable.} In that case, we exploit the fact that $Q$ is AFD with a Cartan subalgebra $A \subset Q$ and hence has lots of central sequences that sit inside the ultraproduct von Neumann subalgebra $A^\omega \subset Q^\omega$. This is a key observation when $Q$ is of type $\rm III$. Using our generalization of the asymptotic orthogonality property in arbitrary free product von Neumann algebras (see Theorem \ref{theorem-AOP} below) and exploiting the recent generalization of Popa's intertwining techniques obtained in \cite{HI15}, we then show that any corner of $A$ must embed with expectation into $M_1$ inside $M$. By exploiting the regularity property of the Cartan inclusion $A \subset Q$ and using a standard maximality argument, we deduce that $Q \subset M_1$.

We point out that our strategy, based on the study of central sequences in the ultraproduct framework {\em via} Popa's asymptotic orthogonality property, works for arbitrary von Neumann algebras. Hence we are able to deal with amenable subfactors $Q \subset M$ in Theorem \ref{thmA} and Corollary \ref{corB} that can possibly be of type ${\rm III}$.

{\bf Assume that $Q$ is nonamenable.} In that case, we use Connes-Takesaki's structure theory \cite{Co72, Ta03} and Popa's deformation/rigidity theory \cite{Po01, Po03, Po06} inside the ultraproduct of the continuous core $(\core_\varphi(M))^\omega$. One of the new features of our proof is to exploit a recent result of Masuda-Tomatsu \cite{MT13} showing that the continuous core of the ultraproduct von Neumann algebra $\core_{\varphi^\omega}(M^\omega)$ sits, as an intermediate von Neumann subalgebra with trace preserving conditional expectations, between $\core_\varphi(M)$ and $(\core_\varphi(M))^\omega$, that is, 
$$\core_\varphi(M) \subset \core_{\varphi^\omega}(M^\omega) \subset (\core_\varphi(M))^\omega.$$
Using Popa's spectral gap rigidity principle and intertwining techniques, we then show that any finite corner of $\core_\varphi(Q)$ must embed into $\core_{\varphi}(M_1)$ inside the ambient continuous core $\core_\varphi(M)$. By a standard maximality argument, we deduce that $\core_\varphi(Q) \subset \core_\varphi(M_1)$ and hence $Q \subset M_1$.

We point out that we do need to pass to the continuous core $\core_\varphi(M)$ in order to make Popa's spectral gap rigidity principle work since we ultimately use Connes's characterization of amenability for {\em finite} von Neumann algebras \cite{Co75}.

We conclude this paper with an appendix in which we give a short proof of an unpublished result due to the second named author showing that Connes's bicentralizer problem has a positive solution for all type ${\rm III_1}$ factors arising as free products of von Neumann algebras.

\tableofcontents

\section{Preliminaries}\label{preliminaries}

For any von Neumann algebra $M$, we will denote by $\mathcal Z(M)$ the centre of $M$, by $\mathcal U(M)$ the group of unitaries in $M$, by $\Ball(M)$ the unit ball of $M$ with respect to the uniform norm $\|\cdot\|_\infty$ and by $(M, \rL^2(M), J^M, \mathfrak P^M)$ the standard form of $M$. More generally, for any linear subspace $\mathcal V \subset M$, we will denote by $\Ball(\mathcal V)$ the unit ball of $\mathcal V$ with respect to the uniform norm $\|\cdot\|_\infty$.

\subsection*{Background on $\sigma$-finite von Neumann algebras}

Let $M$ be any $\sigma$-finite von Neumann algebra with unique predual $M_\ast$ and $\varphi \in M_\ast$ any faithful state. We will write $\|x\|_\varphi = \varphi(x^* x)^{1/2}$ 
for all $x \in M$. Recall that on $\Ball(M)$, the topology given by $\|\cdot\|_\varphi$ 
coincides with the $\sigma$-strong topology. Denote by $\xi_\varphi \in \mathfrak P^M$ the unique canonical implementing vector of $\varphi$. The mapping $M \to \rL^2(M) : x \mapsto x \xi_\varphi$ defines an embedding with dense image such that $\|x\|_\varphi = \|x \xi_\varphi\|_{\rL^2(M)}$ for all $x \in M$.
 
We denote by $\sigma^\varphi$ the modular automorphism group of the state $\varphi$.  The centralizer $M^\varphi$ of the state $\varphi$ is by definition the fixed point algebra of $(M, \sigma^\varphi)$.  The {\em continuous core} of $M$ with respect to $\varphi$, denoted by $\core_\varphi(M)$, is the crossed product von Neumann algebra $M \rtimes_{\sigma^\varphi} \R$.  The natural inclusion $\pi_\varphi: M \to \core_\varphi(M)$ and the unitary representation $\lambda_\varphi: \R \to \core_\varphi(M)$ satisfy the {\em covariance} relation
$$
  \lambda_\varphi(t) \pi_\varphi(x) \lambda_\varphi(t)^*
  =
  \pi_\varphi(\sigma^\varphi_t(x))
  \quad
  \text{ for all }
  x \in M \text{ and all } t \in \R.
$$
Put $\rL_\varphi (\R) = \lambda_\varphi(\R)\dpr$. There is a unique faithful normal conditional expectation $\rE_{\rL_\varphi (\R)}: \core_{\varphi}(M) \to \rL_\varphi(\R)$ satisfying $\rE_{\rL_\varphi (\R)}(\pi_\varphi(x) \lambda_\varphi(t)) = \varphi(x) \lambda_\varphi(t)$. The faithful normal semifinite weight defined by $f \mapsto \int_{\R} \exp(-s)f(s) \, {\rm d}s$ on $\rL^\infty(\R)$ gives rise to a faithful normal semifinite weight $\Tr_\varphi$ on $\rL_\varphi(\R)$ {\em via} the Fourier transform. The formula $\Tr_\varphi = \Tr_\varphi \circ \rE_{\rL_\varphi (\R)}$ extends it to a faithful normal semifinite trace on $\core_\varphi(M)$.

Because of Connes's Radon--Nikodym cocycle theorem \cite[Th\'eor\`eme 1.2.1]{Co72} (see also \cite[Theorem VIII.3.3]{Ta03}), the semifinite von Neumann algebra $\core_\varphi(M)$ together with its trace $\Tr_\varphi$ does not depend on the choice of $\varphi$ in the following precise sense. If $\psi$ is another faithful normal state on $M$, there is a canonical surjective $\ast$-isomorphism
$\Pi_{\varphi,\psi} : \core_\psi(M) \to \core_{\varphi}(M)$ such that $\Pi_{\varphi,\psi} \circ \pi_\psi = \pi_\varphi$ and $\Tr_\varphi \circ \Pi_{\varphi,\psi} = \Tr_\psi$. Note however that $\Pi_{\varphi,\psi}$ does not map the subalgebra $\rL_\psi(\R) \subset \core_\psi(M)$ onto the subalgebra $\rL_\varphi(\R) \subset \core_\varphi(M)$ (and hence we use the symbol $\rL_\varphi(\R)$ instead of the usual $\rL(\R)$).

In order to prove the asymptotic orthogonality property inside arbitrary free product von Neumann algebras (see Theorem \ref{theorem-AOP} below), we will use the following key simple lemma whose proof is similar to \cite[Proposition 2.8]{MU12} using \cite[Theorem 11.1]{HS90}. 

\begin{lem}\label{lemma-centralizer} 
Let $(M, \varphi)$ be any diffuse $\sigma$-finite von Neumann algebra endowed with a faithful normal state. For every $\delta>0$, there exists a faithful normal state $\psi \in M_\ast$ such that $\Vert \varphi - \psi\Vert < \delta$ and such that the centralizer $M^\psi$ is diffuse. 
\end{lem}

\begin{proof} 
Write $M = M_d\oplus M_c$ where $M_d$ is of type I with diffuse center and $M_c$ has no type I direct summand. The above decomposition gives $\varphi = \varphi_d+\varphi_c$. By \cite[Theorem 11.1]{HS90} (which dates back to Connes-St{\o}rmer's transitivity theorem \cite{CS78}), one can find a faithful normal positive linear functional $\varphi'_c \in (M_c)_\ast$ such that $\Vert \varphi'_c\Vert_{(M_c)_\ast} = \Vert \varphi_c\Vert_{(M_c)_\ast}$, $\Vert \varphi_c - \varphi'_c\Vert_{(M_c)_\ast} < \delta$ and $(M_c)^{\varphi'_c}$ is of type ${\rm II_1}$. Put $\psi := \varphi_d+\varphi'_c$ and observe that $\psi \in M_\ast$ is a faithful normal state. Then we have $\Vert \varphi - \psi\Vert_{M_\ast} = \Vert \varphi_c - \varphi'_c\Vert_{(M_c)_\ast} < \delta$ and $\mathcal{Z}(M_d)\oplus(M_c)^{\varphi'_c} \subset M^\psi$. Therefore, the centralizer $M^\psi$ is a diffuse von Neumann subalgebra (see e.g.\ \cite[Theorem IV.2.2.3]{Bl06}). 
\end{proof}

\subsection*{Popa's intertwining techniques}

To fix notation, let $M$ be any $\sigma$-finite von Neumann algebra, $1_A$ and $1_B$ any nonzero projections in $M$, $A\subset 1_AM1_A$ and $B\subset 1_BM1_B$ any von Neumann subalgebras. Popa introduced his powerful {\em intertwining-by-bimodules techniques} in \cite{Po01} in the case when $M$ is finite and more generally in \cite{Po03} in the case when $M$ is endowed with an almost periodic faithful normal state $\varphi$ for which $1_A \in M^\varphi$, $A \subset 1_A M^\varphi 1_A$ and $1_B \in M^\varphi$, $B \subset 1_B M^\varphi 1_B$. It was showed in \cite{HV12,Ue12} that Popa's intertwining techniques extend to the case when $B$ is finite and with expectation in $1_B M 1_B$ and $A \subset 1_A M 1_A$ is any von Neumann subalgebra.

In this paper, we will need the following generalization of \cite[Theorem A.1]{Po01} in the case when $A \subset 1_A M 1_A$ is any finite von Neumann subalgebra with expectation and $B \subset 1_B M 1_B$ is any von Neumann subalgebra with expectation.

\begin{thm}[{\cite[Theorem 4.3]{HI15}}]\label{theorem-intertwining}
Let $M$ be any $\sigma$-finite von Neumann algebra, $1_A$ and $1_B$ any nonzero projections in $M$, $A\subset 1_AM1_A$ and $B\subset 1_BM1_B$ any von Neumann subalgebras with faithful normal conditional expectations $\rE_A : 1_A M 1_A \to A$ and $\rE_B : 1_B M 1_B \to B$ respectively. Assume moreover that $A$ is a finite von Neumann algebra.

Then the following conditions are equivalent:

\begin{enumerate}
\item There exist projections $e \in A$ and  $f \in B$, a nonzero partial isometry $v \in eMf$ and a unital normal $\ast$-homomorphism $\theta : eAe \to fBf$ such that the inclusion $\theta(eAe) \subset fBf$ is with expectation and $av = v \theta(a)$ for all $a \in eAe$.
\item There exist $n \geq 1$, a projection $q \in \mathbf M_n (B)$, a nonzero partial isometry $v\in (1_AM\otimes\mathbf{M}_{1, n}(\C))q$ and a unital normal $\ast$-homomorphism $\pi\colon A \rightarrow q\mathbf{M}_n(B)q$ such that the inclusion $\pi(A) \subset q \mathbf M_n (B) q$ is with expectation and $av = v\pi(a)$ for all $a\in A$. 
\item There exists no net $(w_i)_{i \in I}$ of unitaries in $\mathcal U(A)$ such that $\rE_{B}(b^*w_i a)\rightarrow 0$ $\sigma$-strongly as $i\to\infty$ for all $a,b\in 1_AM1_B$.
\end{enumerate}
If one of the above conditions is satisfied, we will say that $A$ {\em embeds with expectation into} $B$ {\em inside} $M$ and write $A \preceq_M B$.
\end{thm}

Moreover, \cite[Theorem 4.3]{HI15} asserts that when $B \subset 1_B M 1_B$ is a {\em semifinite} von Neumann subalgebra endowed with any fixed faithful normal semifinite trace $\Tr$, then $A \preceq_M B$ if and only if there exist a projection $e \in A$, a $\Tr$-finite projection $f \in B$, a nonzero partial isometry $v \in eMf$ and a unital normal $\ast$-homomorphism $\theta : eAe \to fBf$ such that $av = v \theta(a)$ for all $a \in eAe$. Hence, in that case, the notation $A \preceq_M B$ is consistent with \cite[Proposition 3.1]{Ue12}. In particular, the projection $q \in \mathbf M_n (B)$ in Theorem \ref{theorem-intertwining} (2) is chosen to be finite under the trace $\Tr \otimes \tr_n$, when $B$ is semifinite with any fixed faithful normal semifinite trace $\Tr$. We refer to \cite[Section 4]{HI15} for further details.

We say that a $\sigma$-finite von Neumann algebra $P$ is {\em tracial} if it is endowed with a faithful normal tracial state $\tau$. Following \cite{Jo82, PP84}, a unital inclusion of tracial von Neumann algebras $A \subset (P, \tau)$ has {\em finite Jones index} if $\dim_A(\rL^2(P, \tau)_A) < +\infty$ with the Murray--von Neumann dimension function $\dim_A$ determined by $\tau$. Following \cite[Appendix A]{Va07}, a unital inclusion of tracial von Neumann algebras $A \subset (P, \tau)$ has {\em essentially finite index} if there exists a sequence of nonzero projections $(p_n)_n$ in $A' \cap P$ such that the unital inclusion of tracial von Neumann algebras $Ap_n \subset (p_n P p_n, \frac{\tau(p_n\,\cdot\,p_n)}{\tau(p_n)})$ has finite Jones index for all $n \in \N$ and $p_n \to 1$ $\sigma$-strongly as $n \to \infty$.

We will need the following technical lemma about how the intertwining technique behaves with respect to taking subalgebras of essentially finite index. 

\begin{lem}[{\cite[Lemma 3.9]{Va07}}]\label{lemma-finite-index}
Let $M$ be any $\sigma$-finite von Neumann algebra, $1_P$ and $1_B$ any nonzero projections in $M$, $P\subset 1_PM1_P$ and $B\subset 1_BM1_B$ any von Neumann subalgebras with expectation. Assume moreover that $P$ is a finite von Neumann algebra and $A \subset P$ is a unital von Neumann subalgebra of essentially finite index. Then $A \preceq_M B$ implies $P \preceq_M B$.
\end{lem}

\begin{proof}
This result is \cite[Lemma 3.9]{Va07} when the ambient von Neumann algebra $M$ is {\em finite} and its proof applies {\em mutatis mutandis} to our more general setting.
\end{proof}

We will moreover need the following two technical lemmas about intertwining subalgebras inside continuous cores.

\begin{lem}\label{lemma-intertwining-core}
Let $(M, \varphi)$ be any $\sigma$-finite von Neumann algebra endowed with a faithful normal state. Let $q \in M^\varphi$ be any nonzero projection and $Q \subset qMq$ any von Neumann subalgebra that is globally invariant under the modular automorphism group $\sigma^{\varphi_q}$ of $\varphi_q = \frac{\varphi(q\,\cdot\,q)}{\varphi(q)}$. Denote by $\rE_Q$ the unique $\varphi_q$-preserving conditional expectation from $qMq$ onto $Q$.

Then for every nonzero finite trace projection $p \in \core_\varphi(M)$ and every net $(u_i)_{i\in I}$ in $\Ball(M)$ such that $\lim_i\rE_{Q}(b^* u_i a) = 0$ $\sigma$-strongly for all $a, b \in M q$, we have 
$$\lim_i \|\rE_{\core_{\varphi_{q}}(Q)}(y^* \, p\pi_\varphi(u_i)p \, x)\|_2 = 0, \forall x, y \in p \core_\varphi(M) \pi_\varphi(q).$$

In particular, for any faithful normal state $\psi \in M_\ast$, any nonzero projection $r \in M^\psi$, any von Neumann subalgebra $R \subset rM^\psi r$ satisfying $R \npreceq_M Q$ and any finite trace projection $s \in \rL_\psi(\R)$, we have 
$$\Pi_{\varphi, \psi}(\pi_\psi(R)s) \npreceq_{\core_{\varphi}(M)} \core_{\varphi_{q}}(Q).$$
\end{lem}

\begin{proof}
The proof is essentially contained in \cite[Proposition 2.10]{BHR12} (see also \cite[Proposition 5.3]{HR10}). Simply denote by $\Tr = \Tr_\varphi$ the canonical trace on $\core_\varphi(M)$ and by $\|\cdot\|_2$ the $\rL^2$-norm with respect to $\Tr$. Let $x, y \in \Ball(p \core_\varphi(M)\pi_\varphi(q))$ be any elements. Fix an increasing sequence $(p_m)_m$ of finite trace projections in $\rL_\varphi(\R)$ such that $p_m \to 1$ $\sigma$-strongly. Observe that $p_m \pi_\varphi(q) = \pi_\varphi(q) p_m$ for all $m \in \N$, since $q \in M^\varphi$.

Let $\varepsilon > 0$. Since $\Tr(p) < +\infty$, we may choose $m \in \N$ large enough such that 
\begin{equation}\label{equation-intertwining-core1}
\|px - px p_m\|_2 + \|y^* p - p_m y^* p\|_2 < \frac{\varepsilon}{2}.
\end{equation}
Observe that the unital $\ast$-subalgebra 
$$\mathcal A := \left\{ \sum_{j = 1}^n \pi_\varphi(a_j) \lambda_\varphi(t_j) : n \geq 1, a_1, \dots, a_n \in M, t_1, \dots, t_n \in \R \right\}$$
is $\sigma$-$\ast$-strongly dense in $\core_\varphi(M)$. Using Kaplansky's density theorem and since $\Tr(p_m) < +\infty$, there exist $x_0, y_0 \in \Ball( \mathcal A \pi_\varphi(q))$ such that 
\begin{equation}\label{equation-intertwining-core2}
\| px p_m - x_0 p_m\|_2 + \| p_m y^* p - p_m y_0^*\|_2 < \frac{\varepsilon}{2}.
\end{equation}
Using \eqref{equation-intertwining-core1} and \eqref{equation-intertwining-core2}, for all $i \in I$, we have
\begin{equation}\label{equation-intertwining-core3}
\|\rE_{\core_{\varphi_{q}}(Q)}(y^* \, p\pi_\varphi(u_i)p \, x)\|_2 \leq \|\rE_{\core_{\varphi_{q}}(Q)}(p_m y^*_0 \, \pi_\varphi(u_i) \, x_0p_m)\|_2 + \varepsilon.
\end{equation}

Write $x_0 = \sum_{j = 1}^\ell \pi_\varphi(a_j) \lambda_{\varphi}(t_j)$ and $y_0 = \sum_{k = 1}^n \pi_\varphi(b_k) \lambda_{\varphi}(t'_k)$ for some $a_j, b_k \in Mq$ and $t_j, t'_k \in \R$. Since 
$$\rE_{\core_{\varphi_{q}}(Q)}(p_m y_0^*\, \pi_\varphi(u_i) \, x_0p_m) = \sum_{j, k} p_m \lambda_{\varphi}(t'_k)^* \, \pi_\varphi(\rE_{Q}(b_k^* u_i a_j)) \,  \lambda_{\varphi}(t_j) p_m$$
and since $\lim_i \rE_{Q}(b_k^* u_i a_j) = 0$ $\sigma$-strongly for all $j, k$ and since $\Tr(p_m) < +\infty$, we obtain 
\begin{equation}\label{equation-intertwining-core4}
 \lim_i \|\rE_{\core_{\varphi_{q}}(Q)}(p_m y^*_0 \, \pi_\varphi(u_i) \, x_0p_m)\|_2 = 0.
\end{equation}
Then \eqref{equation-intertwining-core3} and \eqref{equation-intertwining-core4} imply that $\limsup_i \|\rE_{\core_{\varphi_{q}}(Q)}(y^* \, p\pi_\varphi(u_i)p \, x)\|_2 \leq \varepsilon$. Since $\varepsilon > 0$ is arbitrary, we finally obtain 
$$\lim_i \|\rE_{\core_{\varphi_{q}}(Q)}(y^* \, p\pi_\varphi(u_i)p \, x)\|_2 = 0.$$

Next, assume that $\psi \in M_\ast$ is any faithful normal state, $r \in M^\psi$ is any nonzero projection, $R \subset r M^\psi r$ is any von Neumann subalgebra such that $R \npreceq_M Q$ and $s \in \rL_\psi(\R)$ is any nonzero finite trace projection. By Theorem \ref{theorem-intertwining}, there exists a net $(v_i)_{j \in J}$ in $\mathcal U(R)$ such that $\lim_j \|\rE_{Q}(b^* v_j a)\|_\varphi = 0$ for all $a, b \in r M q$. Recall that $\Pi_{\varphi, \psi} \circ \pi_\psi = \pi_\varphi$ and $\Tr_\varphi \circ \Pi_{\varphi, \psi} = \Tr_\psi$. Put $p = \Pi_{\varphi, \psi}(s)$. The first part of the proof implies that $\lim_j \|\rE_{\core_{\varphi_{q}}(Q)}(y^* \, \pi_\varphi(v_j)p \, x)\|_2 = 0$ for all $x, y \in p\pi_\varphi(r)\core_\varphi(M)\pi_\varphi(q)$. Since $\pi_\varphi(v_j)p = \Pi_{\varphi, \psi}(\pi_\psi(v_j)s) \in \mathcal U(\Pi_{\varphi, \psi}(\pi_\psi(R)s))$ for all $j \in J$, we obtain that $\Pi_{\varphi, \psi}(\pi_\psi(R)s) \npreceq_{\core_{\varphi}(M)} \core_{\varphi_{q}}(Q)$ by Theorem \ref{theorem-intertwining}. 
\end{proof}

\begin{lem}\label{lemma-intertwining-core-bis}
Let $M$ be any $\sigma$-finite von Neumann algebra and $\varphi, \psi \in M_\ast$ any faithful states. Let $q \in M^\psi$ be any nonzero projection and $Q \subset qMq$ any diffuse von Neumann subalgebra that is globally invariant under the modular automorphism group $\sigma^{\psi_q}$ of $\psi_q = \frac{\psi(q\,\cdot\,q)}{\psi(q)}$. 
Then for every nonzero finite trace projection $p \in \rL_\psi(\R)$, we have $c_{\psi_q}(Q) \subset c_\psi(M)$ naturally and $$\Pi_{\varphi, \psi}(p \core_{\psi_q}(Q) p) \npreceq_{\core_\varphi (M)} \rL_\varphi(\R).$$
\end{lem}

\begin{proof}
Denote by $z \in \mathcal Z(Q)$ the unique central projection such that $Qz$ is of type ${\rm I}$ and $Qz^\perp$ has no type ${\rm I}$ direct summand. Observe that $z \in M^\psi$, $\pi_\psi(z) \in  \mathcal Z(\core_{\psi_q}(Q))$ and 
$$p \core_{\psi_q}(Q) p = p \core_{\psi_q}(Q) p \, \pi_\psi(z) \oplus p \core_{\psi_q}(Q) p \, \pi_\psi(z^\perp) = p \core_{\psi_z}(Qz) p \oplus p \core_{\psi_{z^\perp}}(Qz^\perp) p.$$

Since $Qz^\perp$ has no type ${\rm I}$ direct summand, $p \core_{\psi_{z^\perp}}(Qz^\perp) p$ has no type ${\rm I}$ direct summand either. This follows from the fact that continuous cores are independent of states or even weights (due to Connes's Radon--Nykodym cocycle theorem  \cite[Th\'eor\`eme 1.2.1]{Co72}) as well as the fact that the continuous core of any type ${\rm III}$ von Neumann algebra must be of type ${\rm II_\infty}$ (see \cite[Theorem XII.1.1]{Ta03}). Hence we have 
$$\Pi_{\varphi, \psi}(p \core_{\psi_q}(Q) p \, \pi_\psi(z^\perp)) \npreceq_{\core_\varphi(M)} \rL_\varphi(\R).$$

Since $Qz$ is of type ${\rm I}$ and diffuse, $\mathcal{Z}(Qz) \subset Q^{\psi_q} z = (Qz)^{\psi_z}$ with $\psi_z := \frac{\psi(z\,\cdot\,z)}{\psi(z)}$ is also diffuse and hence $\mathcal{Z}(Qz) \npreceq_M \C1$. Then Lemma \ref{lemma-intertwining-core} (with letting the $Q$ there be the trivial algebra) implies that 
$$\Pi_{\varphi, \psi}(p \core_{\psi_q}(Q) p \, \pi_\psi(z)) \npreceq_{\core_\varphi(M)} \rL_\varphi(\R).$$
Combining the above two facts, we finally obtain that $\Pi_{\varphi, \psi}(p \core_{\psi_q}(Q) p) \npreceq_{\core_\varphi (M)} \rL_\varphi(\R)$.
\end{proof}

\subsection*{Amalgamated free product von Neumann algebras}

For each $i \in \{ 1, 2 \}$, let $B \subset M_i$ be any inclusion of $\sigma$-finite von Neumann algebras with faithful normal conditional expectation $\rE_i : M_i \to B$. The {\em amalgamated free product} $(M, \rE) = (M_1, \rE_1) \ast_B (M_2, \rE_2)$ is a pair of von Neumann algebra $M$ generated by $M_1$ and $M_2$ and faithful normal conditional expectation $\rE : M \to B$ such that $M_1, M_2$ are {\em freely independent} with respect to $\rE$:
$$\rE(x_1 \cdots x_n) = 0 \; \text{ whenever } \; x_j \in M_{i_j}^\circ \; \text{ and } \; i_1 \neq \cdots \neq  i_{n}.$$
Here and in what follows, we denote by $M_i^\circ = \ker(\rE_i)$. We refer to the product $x_1 \cdots x_n$ where $x_j \in M_{i_j}^\circ$ and $i_1 \neq \cdots \neq i_{n}$ as a {\em reduced word} in $M_{i_1}^\circ \cdots M_{i_n}^\circ$ of {\em length} $n \geq 1$. The linear span of $B$ and of all the reduced words in $M_{i_1}^\circ \cdots M_{i_n}^\circ$ where $n \geq 1$ and $i_1 \neq \cdots \neq i_{n}$ forms a unital $\sigma$-strongly dense $\ast$-subalgebra of $M$. We call the resulting $M$ the \emph{amalgamated free product von Neumann algebra} of $(M_1,\rE_1)$ and $(M_2,\rE_2)$.

When $B = \C 1$, $\rE_i = \varphi_i(\cdot) 1$ for all $i \in \{1, 2\}$ and $\rE = \varphi(\cdot) 1$, we will simply denote by $(M, \varphi) = (M_1, \varphi_1) \ast (M_2, \varphi_2)$ and call the resulting $M$ the {\em free product von Neumann algebra} of $(M_1\varphi_1)$ and $(M_2,\varphi_2)$. 

When $B$ is a semifinite von Neumann algebra with faithful normal semifinite trace $\Tr$ and the weight $\Tr \circ \rE_i$ is tracial on $M_i$ for every $i \in \{1, 2\}$, then the weight $\Tr \circ \rE$ is  tracial on $M$ (see \cite[Proposition 3.1]{Po90} for the finite case and \cite[Theorem 2.6]{Ue98a} for the general case). In particular, $M$ is a semifinite von Neumann algebra. In that case, we will refer to $(M, \rE) = (M_1, \rE_1) \ast_B (M_2, \rE_2)$ as a {\em semifinite} amalgamated free product.

Let $\varphi \in B_\ast$ be any faithful normal state. Then for all $t \in \R$, we have $\sigma_t^{\varphi \circ \rE} = \sigma_t^{\varphi \circ \rE_1} \ast \sigma_t^{\varphi \circ \rE_2}$ (see \cite[Theorem 2.6]{Ue98a}). By \cite[Theorem IX.4.2]{Ta03}, there exists a unique $\varphi\circ \rE$-preserving conditional expectation $\rE_{M_1} : M \to M_1$. Moreover, we have $\rE_{M_1}(x_1 \cdots x_n) = 0$ for all the reduced words $x_1 \cdots x_n$ that contain at least one letter from $M_2^\circ$ (see e.g.\ \cite[Lemma 2.1]{Ue10}). We will denote by $M \ominus M_1 = \ker(\rE_{M_1})$. For more on (amalgamated) free product von Neumann algebras, we refer the reader to \cite{BHR12, Po90, Ue98a, Ue10, Ue12, Vo85, VDN92}.

\begin{lem}\label{lemma-control}
For each $i \in \{ 1, 2 \}$, let $B \subset M_i$ be any inclusion of $\sigma$-finite von Neumann algebras with faithful normal conditional expectations $\rE_i : M_i \to B$. Denote by $(M, \rE) = (M_1, \rE_1) \ast_B (M_2, \rE_2)$ the amalgamated free product. 

Let $\psi \in M_\ast$ be any faithful normal state such that $\psi = \psi \circ \rE_{M_1}$. Let $(u_j)_{j\in J}$ be any net in $\Ball((M_1)^\psi)$ such that $\lim_j \rE_1(b^* u_j a) = 0$ $\sigma$-strongly for all $a, b \in M_1$. Then for all $x, y \in M \ominus M_1$, we have that $\lim_j \rE_{M_1}(y^* u_j x) = 0$ $\sigma$-strongly.
\end{lem}

\begin{proof}
We first prove the $\sigma$-strong convergence when $x, y \in M_1 M_2^\circ \cdots M_2^\circ M_1$ are {\em words} of the form $x = a x' c$ and $y = b y' d$ with $a, b, c, d \in M_1$ and $x', y' \in M_2^\circ \cdots M_2^\circ$. By freeness with amalgamation over $B$, for all $n \in \N$, we have
$$\rE_{M_1}(y^* u_j x) = \rE_{M_1}(d^*y'^* \, b^* u_j a \, x'c) = \rE_{M_1}(d^*y'^* \, \rE_1(b^* u_j a)  \,  x'c).$$
Since $\lim_j\rE_1(b^* u_j a) = 0$ $\sigma$-strongly, we have $\lim_j\rE_{M_1}(y^* u_j x) = 0$ $\sigma$-strongly.

Recall that $\psi = \psi \circ \rE_{M_1}$. We next prove the $\sigma$-strong convergence when $x \in M \ominus M_1$ is any element and $y \in M_1 M_2^\circ \cdots M_2^\circ M_1$ is any {\em word} as above. Indeed, we may choose a sequence $(x_k)_k$, where each $x_k$ is a finite linear combination of {\em words} in $M_1 M_2^\circ \cdots M_2^\circ M_1$, and such that $\lim_{k\to\infty} \|x - x_k\|_\psi = 0$. Then by triangle inequality, for all $j \in J$ and $k \in \N$, we have
\begin{align*}
\|\rE_{M_1}(y^* u_j x)\|_\psi &\leq \|\rE_{M_1}(y^* u_j x_k)\|_\psi + \|\rE_{M_1}(y^* u_j (x - x_k))\|_\psi \\
&\leq \|\rE_{M_1}(y^* u_j x_k)\|_\psi + \|y^* u_n (x - x_k)\|_\psi \\
&\leq \|\rE_{M_1}(y^* u_j x_k)\|_\psi + \|y\|_\infty \|x - x_k\|_\psi.
\end{align*}
The first part of the proof implies that $\limsup_j \|\rE_{M_1}(y^* u_j x)\|_\psi  \leq \|y\|_\infty \|x - x_k\|_\psi$ for all $k \in \N$ and hence $\lim_j \|\rE_{M_1}(y^* u_j x)\|_\psi = 0$.

Recall that $\psi = \psi \circ \rE_{M_1}$ and hence $\sigma_t^\psi(M_1) = M_1$ for all $t \in \R$. 
We next prove the $\sigma$-strong convergence when $x \in M \ominus M_1$ is any analytic element with respect to the modular automorphism group $\sigma^\psi$ and $y \in M \ominus M_1$ is any element. Indeed, we may choose a sequence $(y_k)_k$, where each $y_k$ is a finite linear combination of {\em words} in $M_1 M_2^\circ \cdots M_2^\circ M_1$, and such that $\lim_{k\to\infty} \|y^* - y_k^*\|_\psi = 0$. Then by triangle inequality, for all $j \in J$ and all $k \in \N$, we have 
\begin{align*}
\|\rE_{M_1}(y^* u_j x)\|_\psi &\leq \|\rE_{M_1}(y^*_k u_j x)\|_\psi + \|\rE_{M_1}((y^*-y^*_k) u_n x)\|_\psi \\
&\leq \|\rE_{M_1}(y^*_k u_j x)\|_\psi + \|(y^* - y_k^*) u_j x \|_\psi \\
&= \|\rE_{M_1}(y^*_k u_j x)\|_\psi +  \|J^M \sigma_{{\rm i}/2}^\psi(x)^* u_j^* J^M(y^* - y_k^*) \|_\psi \\
&= \|\rE_{M_1}(y^*_k u_j x)\|_\psi +  \|\sigma_{{\rm i}/2}^\psi(x)\|_\infty \|y^* - y_k^* \|_\psi.
\end{align*}
The second part of the proof implies that $\limsup_j \|\rE_{M_1}(y^* u_j x)\|_\psi  \leq \|\sigma_{{\rm i}/2}^\psi(x)\|_\infty \|y^* - y_k^* \|_\psi$ for all $k \in \N$ and hence $\lim_j \|\rE_{M_1}(y^* u_j x)\|_\psi = 0$.

We finally prove the $\sigma$-strong convergence when $x, y \in M \ominus M_1$ are any elements. Indeed, we may choose a sequence $(x_k)_k$ in $M \ominus M_1$ of analytic elements with respect to the modular automorphism group $\sigma^\psi$ such that $\lim_{k\to\infty} \|x - x_k\|_\psi = 0$. Then by triangle inequality, for all $j \in J$ and all $k \in \N$, we have 
\begin{align*}
\|\rE_{M_1}(y^* u_j x)\|_\psi &\leq \|\rE_{M_1}(y^* u_j x_k)\|_\psi + \|\rE_{M_1}(y^* u_j (x - x_k))\|_\psi \\
&\leq \|\rE_{M_1}(y^* u_j x_k)\|_\psi + \|y^* u_j (x - x_k)\|_\psi \\
&\leq \|\rE_{M_1}(y^* u_j x_k)\|_\psi + \|y\|_\infty \|x - x_k\|_\psi.
\end{align*}
The third part of the proof implies that $\limsup_j \|\rE_{M_1}(y^* u_j x)\|_\psi  \leq \|y\|_\infty \|x - x_k\|_\psi$ for all $k \in \N$ and hence $\lim_j \|\rE_{M_1}(y^* u_j x)\|_\psi = 0$. This finishes the proof of Lemma \ref{lemma-control}.
\end{proof}

The next proposition about controlling the {\em (quasi)-normalizer} of diffuse subalgebras inside free product von Neumann algebras will be very useful in the proof of Theorem \ref{thmA}. This is a variant of \cite[Theorem 1.1]{IPP05} and \cite[Proposition 3.3]{Ue12}, but the proof uses an idea of \cite[Lemma D3]{Va06} and the previous lemma crucially. We point out that the first assertion also generalizes \cite[Corollary 3.2]{Ue10} (with $n=1$, $\pi(x) = uxu^*$ and $v=u$ for $u \in \mathcal{U}(A' \cap M)$ or $\mathcal{N}_M(A)$). A more general, unified statement seems possible in the framework of amalgamated free products because the previous lemma is quite general, but the statements below fit the later use.

\begin{prop}\label{proposition-control} 
For each $i \in \{1, 2\}$, let $(M_i, \varphi_i)$ be any $\sigma$-finite von Neumann algebra endowed with a faithful normal state. Denote by $(M, \varphi) = (M_1, \varphi_1) \ast (M_2, \varphi_2)$ the free product. 
\begin{enumerate}
\item Let $1_Q \in M_1$ be any nonzero projection and $Q \subset 1_Q M_1 1_Q$ any diffuse von Neumann subalgebra with expectation. For every $n \geq 1$, every (not necessarily unital) normal $\ast$-homomorphism $\pi : Q \to \mathbf M_n(M_1)$ and every nonzero partial isometry $v \in (1_QM \otimes \mathbf M_{1, n}(\C))\pi(1_Q)$ such that $x v = v \pi(x)$ for all $x \in Q$, we have 
$$v \in (1_QM_1 \otimes \mathbf M_{1, n}(\C))\pi(1_Q).$$

\item Let $1_A \in M$ be any nonzero projection and $A \subset 1_A M 1_A$ any diffuse von Neumann subalgebra with expectation. For every $n \geq 1$, every (not necessarily unital) normal $\ast$-homomorphism $\pi : A \to \mathbf M_n(M_1)$ such that the inclusion $\pi(A) \subset \pi(1_A) \mathbf M_n(M_1) \pi(1_A)$ is with expectation and every nonzero partial isometry $v \in (1_AM \otimes \mathbf M_{1, n}(\C))\pi(1_A)$ such that $a v = v \pi(a)$ for all $a \in A$, we have 
$$v^*v \in \pi(1_A)\mathbf M_n(M_1)\pi(1_A) \quad \text{and} \quad v^* \mathcal N_{1_AM1_A}(A)\dpr v \subset v^*v\mathbf M_n(M_1)v^*v.$$
\end{enumerate}
\end{prop}

\begin{proof} $(1)$ As in the proof of Lemma \ref{lemma-centralizer} and since $Q \subset 1_Q M_1 1_Q$ is with expectation, we may choose a faithful normal state $\psi \in M_\ast$ such that $\psi = \psi \circ \rE_{M_1}$, $1_Q \in (M_1)^\psi$, $Q \subset 1_Q M 1_Q$ is globally invariant under the modular automorphism group $\sigma^{\psi_Q}$ and $Q^{\psi_{Q}} \subset 1_Q (M_1)^\psi 1_Q$ is diffuse where $\psi_Q := \frac{\psi(1_Q \, \cdot \, 1_Q)}{\psi(1_Q)}$.

Let $n, \pi, v$ as in the statement. Denote by $\tr_n$ the canonical normalized trace on $\mathbf M_n(\C)$ and write $v = [v_1 \cdots v_n] \in (1_QM \otimes \mathbf M_{1, n}(\C))\pi(1_Q)$. For all $x \in Q$, since $x \, v = v \, \pi(x)$, we have 
$$x \, \rE_{\mathbf M_n(M_1)}( v) = \rE_{\mathbf M_n(M_1)}(x \, v)  = \rE_{\mathbf M_n(M_1)}(v \, \pi(x)) = \rE_{\mathbf M_n(M_1)}(v) \, \pi(x)$$
and hence
\begin{equation}\label{equation-control1}
x \, (v -  \rE_{\mathbf M_n(M_1)}( v)) = (v -  \rE_{\mathbf M_n(M_1)}(v)) \, \pi(x).
\end{equation}
Put $w := v - \rE_{\mathbf M_n(M_1)}(v) \in (1_Q(M \ominus M_1) \otimes \mathbf M_{1, n}(\C))\pi(1_Q)$ and write $w = [w_1 \cdots w_n]$ with $w_1, \dots, w_n \in 1_Q(M \ominus M_1)$.
 Fix a sequence of unitaries $(u_k)_k$ in $\mathcal U(Q^{\psi_Q})$ such that $\lim_{k\to\infty} u_k = 0$ $\sigma$-weakly. By Lemma \ref{lemma-control}, we have 
\begin{equation}\label{equation-control2}
\lim_{k\to\infty} \|\rE_{\mathbf M_n(M_1)}(w^* u_k w)\|_{\psi \otimes \tr_n}^2 = \lim_{k\to\infty} \sum_{i, j = 1}^n \|\rE_{M_1}(w^*_i u_k w_j)\|_\psi^2 = 0.
\end{equation}
Using \eqref{equation-control1} and \eqref{equation-control2} and since $\pi(u_k) \in \mathcal U(\pi(Q))$ and $w^*w \in \pi(Q)' \cap \pi(1_Q) \mathbf M_n(M) \pi(1_Q)$, we have 
\begin{align*}
\|\rE_{\mathbf M_n(M_1)}(w^*w)\|_{\psi \otimes \tr_n} &= \limsup_{k\to\infty} \| \pi(u_k) \, \rE_{\mathbf M_n(M_1)}(w^*w) \|_{\psi \otimes \tr_n} \\
&= \limsup_{k\to\infty} \|\rE_{\mathbf M_n(M_1)}(\pi(u_k) \, w^*w )\|_{\psi \otimes \tr_n} \\
&= \limsup_{k\to\infty} \|\rE_{\mathbf M_n(M_1)}(w^*w \, \pi(u_k))\|_{\psi \otimes \tr_n} \\
&= \lim_{k\to\infty} \|\rE_{\mathbf M_n(M_1)}(w^* \, u_k \, w)\|_{\psi \otimes \tr_n} \\
& = 0.
\end{align*}
This implies that $w^* w = 0$ and hence $w = 0$. Thus $v = \rE_{\mathbf M_n(M_1)}(v) \in (1_QM_1 \otimes \mathbf M_{1, n}(\C))\pi(1_Q)$.

$(2)$ We will be working inside the amalgamated free product von Neumann algebra
$$\mathbf M_n(M) = (\mathbf M_n(M_1), \varphi_1 \otimes \id_n) \ast_{\mathbf M_n(\C)} (\mathbf M_n(M_2), \varphi_2 \otimes \id_n),$$
and substitute formula \eqref{equation-contro13} below for the assumption of item $(1)$ that $xv = v\pi(x)$ for all $x \in Q$. 

Since $\pi(A) \subset \pi(1_A) \mathbf M_n(M_1) \pi(1_A)$ is a diffuse von Neumann subalgebra with expectation, we may choose, as in the proof of item $(1)$, a faithful normal state $\psi \in \mathbf M_n(M)_\ast$ such that $\psi = \psi \circ \rE_{\mathbf M_n(M_1)}$, $\pi(1_A) \in \mathbf M_n(M_1)^\psi$ and $\pi(A) \cap \pi(1_A) \mathbf M_n(M_1)^\psi \pi(1_A)$ is diffuse. Fix a sequence of unitaries $(u_k)_k$ in $\pi(A) \cap \pi(1_A) \mathbf M_n(M_1)^\psi \pi(1_A)$ such that $\lim_{k\to\infty}u_k = 0$ $\sigma$-weakly. For each $k \in \N$, we may write $u_k = \pi(a_k)$ with a unitary $a_k \in A$.

Let now $x \in \mathcal{N}_{1_A M 1_A}(A)$ be any normalizing unitary element. Then for all $a \in A$, we have
\begin{equation}\label{equation-contro13}
v^* x v \, \pi(a) = v^* x a v = v^* (x a x^*) x v = \pi(xax^*) \, v^*x v,
\end{equation}
and hence, as in the proof of item $(1)$, for every $k \in\N$ we have
\begin{equation}\label{equation-control4}
(v^* x v - \rE_{\mathbf M_n(M_1)}(v^* x v)) \, u_k = \pi(xa_k x^*) \, (v^* x v - \rE_{\mathbf M_n(M_1)}(v^* x v)).
\end{equation}
Put $w := v^* x v - \rE_{\mathbf M_n(M_1)}(v^* x v) \in \pi(1_A)(\mathbf M_n(M) \ominus \mathbf M_n(M_1))\pi(1_A)$. Using \eqref{equation-control4} and Lemma \ref{lemma-control} and since $\pi(xa_k x^*) \in \mathcal U(\pi(A))$, we obtain, as in the proof of item $(1)$, that 
$$
\|\rE_{\mathbf M_n(M_1)}(w w^* )\|_\psi = \lim_{k\to\infty} \| \rE_{\mathbf M_n(M_1)}( w \, u_k \, w^* )\|_\psi = 0,
$$
implying that $v^*xv = \rE_{\mathbf M_n(M_1)}(v^* x v) \in  \mathbf M_n(M_1)$ and the desired assertion is immediate.
\end{proof}

\subsection*{Ultraproduct von Neumann algebras}

Let $M$ be any $\sigma$-finite von Neumann algebra. Define
\begin{align*}
\mathcal I_\omega(M) &= \left\{ (x_n)_n \in \ell^\infty(\N, M) : x_n \to 0\ \text{$\ast$-strongly as } n \to \omega \right\}, \\
\mathcal M^\omega(M) &= \left \{ (x_n)_n \in \ell^\infty(\N, M) :  (x_n)_n \, \mathcal I_\omega(M) \subset \mathcal I_\omega(M) \text{ and } \mathcal I_\omega(M) \, (x_n)_n \subset \mathcal I_\omega(M)\right\}.
\end{align*}

We have that the {\em multiplier algebra} $\mathcal M^\omega(M)$ is a C$^*$-algebra and $\mathcal I_\omega(M) \subset \mathcal M^\omega(M)$ is a norm closed two-sided ideal. Following \cite{Oc85}, we define the {\em ultraproduct von Neumann algebra} $M^\omega$ by $M^\omega = \mathcal M^\omega(M) / \mathcal I_\omega(M)$. We denote the image of $(x_n)_n \in \mathcal M^\omega(M)$ by $(x_n)^\omega \in M^\omega$. 

For all $x \in M$, the constant sequence $(x)_n$ lies in the multiplier algebra $\mathcal M^\omega(M)$. We will then identify $M$ with $(M + \mathcal I_\omega(M))/ \mathcal I_\omega(M)$ and regard $M \subset M^\omega$ as a von Neumann subalgebra. The map $\rE_\omega : M^\omega \to M : (x_n)^\omega \mapsto \sigma \text{-weak} \lim_{n \to \omega} x_n$ is a faithful normal conditional expectation. For every faithful normal state $\varphi \in M_\ast$, the formula $\varphi^\omega = \varphi \circ \rE_\omega$ defines a faithful normal state on $M^\omega$. Observe that $\varphi^\omega((x_n)^\omega) = \lim_{n \to \omega} \varphi(x_n)$ for all $(x_n)^\omega \in M^\omega$.

Let $Q \subset M$ be any von Neumann subalgebra with faithful normal conditional expectation $\rE_Q : M \to Q$. Choose a faithful normal state $\varphi \in M_\ast$ such that $\varphi = \varphi \circ \rE_Q$. We have $\ell^\infty(\N, Q) \subset \ell^\infty(\N, M)$, $\mathcal I_\omega(Q) \subset \mathcal I_\omega(M)$ and $\mathcal M^\omega(Q) \subset \mathcal M^\omega(M)$. We will then identify $Q^\omega = \mathcal M^\omega(Q) / \mathcal I_\omega(Q)$ with $(\mathcal M^\omega(Q) + \mathcal I_\omega(M)) / \mathcal I_\omega(M)$ and regard $Q^\omega \subset M^\omega$ as a von Neumann subalgebra. Observe that the norm $\|\cdot\|_{(\varphi |_Q)^\omega}$ on $Q^\omega$ is the restriction of the norm $\|\cdot\|_{\varphi^\omega}$ to $Q^\omega$. Observe moreover that $(\rE_Q(x_n))_n \in \mathcal I_\omega(Q)$ for all $(x_n)_n \in \mathcal I_\omega(M)$ and $(\rE_Q(x_n))_n \in \mathcal M^\omega(Q)$ for all $(x_n)_n \in \mathcal M^\omega(M)$. Therefore, the mapping $\rE_{Q^\omega} : M^\omega \to Q^\omega : (x_n)^\omega \mapsto (\rE_Q(x_n))^\omega$ is a well-defined conditional expectation satisfying $\varphi^\omega \circ \rE_{Q^\omega} = \varphi^\omega$. Hence, $\rE_{Q^\omega} : M^\omega \to Q^\omega$ is a faithful normal conditional expectation.

Put $\mathcal H = \rL^2(M)$. The {\em ultraproduct Hilbert space} $\mathcal H^\omega$ is defined to  be the quotient of $\ell^\infty(\N, \mathcal H)$ by the subspace consisting in sequences $(\xi_n)_n$ satisfying $\lim_{n \to \omega} \|\xi_n\|_{\mathcal H} = 0$. We denote the image of $(\xi_n)_n \in \ell^\infty(\N, \mathcal H)$ by $(\xi_n)_\omega \in \mathcal H^\omega$. The inner product space structure on the Hilbert space $\mathcal H^\omega$ is defined by $\langle (\xi_n)_\omega, (\eta_n)_\omega\rangle_{\mathcal H^\omega} = \lim_{n \to \omega} \langle \xi_n, \eta_n\rangle_{\mathcal H}$. The standard Hilbert space $\rL^2(M^\omega)$ can be embedded into $\mathcal H^\omega$ as a closed subspace {\em via} the mapping $\rL^2(M^\omega) \to \mathcal H^\omega : (x_n)^\omega \xi_{\varphi^\omega} \mapsto (x_n \xi_\varphi)_\omega$. For more on ultraproduct von Neumann algebras, we refer the reader to \cite{AH12, Oc85}.

In Section \ref{proofs-main-results}, we will need the following well-known fact about ultraproducts of semifinite von Neumann algebras. Let $(M, \Tr)$ be any semifinite $\sigma$-finite von Neumann endowed with a faithful normal semifinite trace. Then the ultraproduct von Neumann algebra $M^\omega$ is semifinite and the weight $\Tr \circ \rE_\omega$ is tracial on $M^\omega$ (see \cite[Lemma 4.26]{AH12}).

In Appendix \ref{appendix}, we will need the following result about the centralizer $(M^\omega)^{\varphi^\omega}$ of the ultraproduct state $\varphi^\omega$.

\begin{prop}\label{proposition-centralizer-ultraproduct}
Let $(M, \varphi)$ be any $\sigma$-finite von Neumann algebra endowed with a faithful normal state and $\omega \in \beta(\N) \setminus \N$ any nonprincipal ultrafilter.
\begin{enumerate}
\item If $M \neq \C 1$, then $(M^\omega)^{\varphi^\omega} \neq \C 1$.
\item If $M$ is diffuse, then $(M^\omega)^{\varphi^\omega}$ is diffuse.
\end{enumerate}
\end{prop}

\begin{proof}
$(1)$ Assume that $M \neq \C 1$. If $M^\varphi \neq \C 1$, then we also have $(M^\omega)^{\varphi^\omega} \neq \C 1$ since $M^\varphi \subset (M^\omega)^{\varphi^\omega}$. If $M^\varphi = \C 1$, then $M$ is a type ${\rm III_1}$ factor by \cite[Lemma 5.3]{AH12}. By \cite[Theorem 4.20]{AH12}, $(M^\omega)^{\varphi^\omega}$ is a type ${\rm II_1}$ factor and hence $(M^\omega)^{\varphi^\omega} \neq \C1$.

$(2)$ Fix a sequence $(z_n)_n$ of central projections in $\mathcal Z(M)$ such that $\sum_n z_n = 1$, $M z_0$ has a diffuse center and $M z_n$ is a diffuse factor for every $n \geq 1$. Observe that $\mathcal Z(M z_0) \subset M^\varphi z_0$ and hence $M^\varphi z_0$ is diffuse. Next, fix $n \geq 1$ such that $z_n \neq 0$ and put $\varphi_{z_n} = \frac{\varphi (z_n\,\cdot\,z_n)}{\varphi(z_n)} \in (M z_n)_\ast$. If $M z_n$ is a semifinite factor, then $M^\varphi z_n = (M z_n)^{\varphi_{z_n}}$ is diffuse. If $M z_n$ is a type ${\rm III_\lambda}$ factor, with $0 \leq \lambda < 1$, then $M^\varphi z_n = (M z_n)^{\varphi_{z_n}}$ is diffuse by  \cite[Th\'eor\`eme 4.2.1 and Th\'eor\`eme 5.2.1]{Co72}. If $M z_n$ is a type ${\rm III_1}$ factor, then $(M^\omega)^{\varphi^\omega}z_n = ((Mz_n)^\omega)^{\varphi_{z_n}^\omega}$ is a type ${\rm II_1}$ factor by \cite[Theorem 4.20]{AH12}. We finally obtain that $(M^\omega)^{\varphi^\omega}z_n = ((Mz_n)^\omega)^{\varphi_{z_n}^\omega}$ is diffuse for all $n$ and hence $(M^\omega)^{\varphi^\omega}$ is diffuse.
\end{proof}

\section{Asymptotic orthogonality property}\label{section-AOP}

The phenomenon of {\em asymptotic orthogonality property} inside free group factors was discovered by Popa in his seminal work \cite[Lemma 2.1]{Po83}. The main result of this section is the following optimal asymptotic orthogonality property result inside arbitrary free product von Neumann algebras. To fix notation, for each $i \in \{1, 2\}$, let $(M_i, \varphi_i)$ be any $\sigma$-finite von Neumann algebra endowed with a faithful normal state. Denote by $(M, \varphi) = (M_1, \varphi_1) \ast (M_2, \varphi_2)$ the free product. As usual, denote by $\rE_{M_1} : M \to M_1$ the unique $\varphi$-preserving conditional expectation. Let $Q \subset M_1$ be any diffuse von Neumann subalgebra with expectation. Fix a faithful state $\psi \in M_\ast$ such that $\sigma_t^\psi(Q) = Q$ and $\sigma_t^\psi(M_1) = M_1$ for all $t \in \R$. Observe that $\psi = \psi \circ \rE_{M_1}$.

Theorem \ref{theorem-AOP} below is a simultaneous generalization of \cite[Proposition 3.5]{Ue10} (which only deals with $y\in\ker(\varphi_2)$) and \cite[Theorem 3.1]{Ho14} (which requires the centralizer $(M_1)^{\varphi_1}$ to be diffuse).

\begin{thm}\label{theorem-AOP} 
Keep the same notation as above. For all $x \in  Q'\cap M^\omega$ and all $y,z \in M \ominus M_1$, the vectors 
$$y(x-\rE_{M_1^\omega}(x))\xi_{\psi^\omega}, \; (y\rE_{M_1^\omega}(x) - \rE_{M_1^\omega}(x)z)\xi_{\psi^\omega} \; \text{ and } \; (\rE_{M_1^\omega}(x)-x)z\xi_{\psi^\omega}$$
are mutually orthogonal in the standard Hilbert space $\rL^2(M^\omega)$ where $\xi_{\psi^\omega} \in \mathfrak P^{M^\omega}$ is the canonical representing vector of the ultraproduct state $\psi^\omega$.  
\end{thm}

\begin{proof}
The proof of Theorem \ref{theorem-AOP} is a reconstruction of \cite[Theorem 3.1]{Ho14} and the new input is the `state replacement' procedure developed in \cite{Ue10}.  

Let $(M^\omega, \rL^2(M^\omega), J^{M^\omega}, \mathfrak{P}^{M^\omega})$ be the standard form of the ultraproduct von Neumann algebra $M^\omega$, which is known to be obtained from the standard form $(M,\rL^2(M),J^M,\mathfrak{P}^M)$ of the original von Neumann algebra $M$ in a rather canonical fashion (see \cite[Corollary 3.27]{AH12}). It suffices to prove, instead of the original assertion, that, for all $z' \in M \ominus M_1$ with the given $x,y$ in the original assertion, the vectors
$$
y(x-\rE_{M_1^\omega}(x))\xi_{\psi^\omega}, \; (y\rE_{M_1^\omega}(x) - J^{M^\omega} z' J^{M^\omega}\rE_{M_1^\omega}(x))\xi_{\psi^\omega} \; \text{ and } \; J^{M^\omega} z' J^{M^\omega}(\rE_{M_1^\omega}(x)-x)\xi_{\psi^\omega}
$$ 
are mutually orthogonal in the standard Hilbert space $\rL^2(M^\omega)$. In fact, by a standard approximation argument we may and do assume that the given $z$ in the original assertion is analytic with respect to the modular automorphism group $\sigma^\psi$. By \cite[Theorem 4.1]{AH12} together with \cite[Lemma VIII.3.18 (ii)]{Ta03}, we have 
\begin{align*}
\rE_{M_1^\omega}(x)z\xi_{\psi^\omega} 
&= 
J^{M^\omega} \sigma_{{\rm i}/2}^\psi(z)^* J^{M^\omega} \rE_{M_1^\omega}(x)\xi_{\psi^\omega}, \\
(\rE_{M_1^\omega}(x)-x)z\xi_{\psi^\omega} 
&=
J^{M^\omega} \sigma_{{\rm i}/2}^\psi(z)^* J^{M^\omega} (\rE_{M_1^\omega}(x)-x)\xi_{\psi^\omega},
\end{align*} 
so that the above new assertion immediately gives the desired one.

For all $i \in \{1, 2\}$, denote by $A_i \subset M_i$ the $\sigma$-weakly dense unital $\ast$-subalgebra of all the analytic elements in $M_i$ with respect to the modular automorphism group $\sigma^{\varphi_i}$ and write $A_i^\circ := A_i \cap M_i^\circ$ with the standard notation $M_i^\circ := \ker(\varphi_i)$. As in the proof of \cite[Theorem 3.1]{Ho14}, we may and will assume that the elements $y$ and $z'$ are analytic with respect to the modular automorphism group $\sigma^\varphi$ and $y$ and $\sigma_{{\rm i}/2}^\varphi(z')^*$ are finite sums of reduced words $w_1,\dots,w_\ell$ and $w'_1,\dots,w'_{\ell'}$ in $A_1 A_2^\circ \cdots A_2^\circ A_1$, respectively. {\it Unlike usual, we call an element in $M_1 M_2^\circ \cdots M_2^\circ M_1$ a reduced word in what follows.} 

Let $V$ be the finite dimensional subspace of $M_1$ obtained by looking at the letters coming from $A_1^\circ \cup \{1\}$ appearing in $y$ in the same fashion as in the proof of \cite[Theorem 3.1]{Ho14}. Namely, $V$ is the linear span of the following $A_1$-letters: 
\begin{itemize}
\item the leftmost $A_1$-letters of the reduced words $w_i, w_i^*$, $1 \leq i \leq \ell$; 
\item the rightmost $A_1$-letters of the reduced words $w'_i, \sigma_{-{\rm i}}^{\varphi}(w'_i{}^*)$, $1 \leq i \leq \ell'$;
\item the leftmost $A_1$-letters of all the reduced words appearing in the elements $w_i^* w_j$, $1 \leq i,j \leq \ell$;
\item the rightmost $A_1$-letters of all the reduced words appearing in the elements $w'_i \sigma_{-{\rm i}}^{\varphi}(w'_j{}^*)$, $1 \leq i,j \leq \ell'$. 
\end{itemize}
Choose an orthonormal basis $e_1,\dots,e_{m}$ of $V$ with respect to the inner product $(a|b)_{\varphi_1} := \varphi_1(b^* a)$ on $M_1$. Denote by $W$ the range of the mapping $a \in M_1 \mapsto a - \sum_{i=1}^m (a|e_i)_{\varphi_1}e_i \in M_1$. It follows that $M_1 = V + W$ is an orthogonal decomposition with respect to the inner product $(\cdot|\cdot)_{\varphi_1}$ defined on $M_1$ as above.

Let $\mathfrak{H}$ be the direct sum of all the alternating tensor products in $\rL^2(M_1)^\circ$ and $\rL^2(M_2)^\circ$ starting and ending with  $\rL^2(M_2)^\circ$. Here $\rL^2(M_i)^\circ$ denotes the orthogonal complement of the canonical representing vector $\xi_{\varphi_i} \in \mathfrak P^{M_i}$ of the given state $\varphi_i$. Thanks to $\ast{\rm-alg}(M_1,M_2) = M_1 + \mathrm{span}(M_1 M_2^\circ\cdots M_2^\circ M_1)$ together with the formula of modular conjugation (see \cite[Proposition II-C]{Ue98a}), the standard Hilbert space $\rL^2(M)$ is naturally identified with $\rL^2(M_1) \oplus \rL^2(M_1)\otimes\mathfrak{H}\otimes \rL^2(M_1)$ as $M_1$-$M_1$-bimodules.  Decompose $\rL^2(M_1)\otimes\mathfrak{H}\otimes \rL^2(M_1)$ into three subspaces $\mathcal{K}_1, \mathcal{K}_2, \mathcal{L}$ defined by 
\begin{align*}
\mathcal{K}_1 &:= (V\xi_{\varphi_1})\otimes\mathfrak{H}\otimes \rL^2(M_1), \\
\mathcal{K}_2 &:= (\overline{W\xi_{\varphi_1}})\otimes\mathfrak{H}\otimes(V\xi_{\varphi_1}), \\\mathcal{L} &:= (\overline{W\xi_{\varphi_1}})\otimes\mathfrak{H}\otimes(\overline{W\xi_{\varphi_1}}).
\end{align*}
It is clear that these subspaces are generated by 
\begin{align*}
&V M_2^\circ \cdots M_2^\circ M_1 \xi_\varphi,\\ 
&W M_2^\circ \cdots M_2^\circ V \xi_\varphi,\\ &W M_2^\circ \cdots M_2^\circ W \xi_\varphi,
\end{align*} 
respectively, in $\rL^2(M)$, where $\xi_\varphi \in \mathfrak P^M$ is the canonical representing vector of the free product state $\varphi$. Remark that the direct summand $\rL^2(M_1)$ in $\rL^2(M)$ is given by $\overline{M_1 \xi_\varphi} = \overline{M_1\xi_\psi}$ thanks to $\psi\circ \rE_{M_1} = \psi$ (see e.g.~\cite[Appendix I]{Ko88}).

Let $\delta>0$ be arbitrarily chosen. By Lemma \ref{lemma-centralizer}, choose a faithful state $\phi_1 \in Q_\ast$ such that $\|\psi|_Q - \phi_1\| < \delta$ and $Q^{\phi_1}$ is diffuse. Denote by $\rE_Q^{M_1} : M_1 \to Q$ the unique $\psi$-preserving conditional expectation and put $\phi := \phi_1\circ \rE_Q^{M_1} \circ  \rE_{M_1}$. Then we have $\phi = \phi \circ \rE_{M_1}$, $Q^\phi$ is diffuse and $\Vert \psi - \phi\Vert = \Vert \psi|_Q - \phi_1\Vert < \delta$ so that the canonical representing vectors $\xi_{\psi}, \xi_\phi \in \mathfrak P^M$ of the states $\psi, \phi$ satisfy $\Vert \xi_\psi - \xi_{\phi}\Vert_{\rL^2(M)} < \delta^{1/2}$ by the Araki-Powers-St{\o}rmer inequality (see \cite[Theorem IX.1.2 (iv)]{Ta03}). In what follows, we denote by $P_\mathcal{X}$ the orthogonal projection from $\rL^2(M)$ onto a (closed) subspace $\mathcal{X}$.

Let $(x_n)_n \in \mathcal M^\omega(M)$ such that $x = (x_n)^\omega$ with $C:= \sup_n \Vert x_n\Vert_\infty$. Then for all $n \in \N$ and all $i \in \{1, 2\}$, we have
\begin{equation}\label{Eq1}
\Vert P_{\mathcal{K}_i}x_n\xi_\psi\Vert_{\rL^2(M)} < C\delta^{1/2} + \Vert P_{\mathcal{K}_i}x_n\xi_\phi\Vert_{\rL^2(M)}. 
\end{equation}
For a while, we will be working with $\Vert P_{\mathcal{K}_i}x_n\xi_\phi\Vert_{\rL^2(M)}$ by the same method used in the proof of \cite[Theorem 3.1]{Ho14}. Since $Q^{\phi}$ is diffuse, we can choose a unitary $u \in \mathcal U(Q^{\phi})$ such that $\lim_{k \to \pm\infty} u^k = 0$ $\sigma$-weakly. Consider the unitary transformation $T : \rL^2(M) \to \rL^2(M) : \xi \to  u \, J^M u J^M \xi =: u\cdot\xi\cdot u^*$. Observe that since $u\in \mathcal U(M^\phi)$ and hence $[u,\xi_\phi] = 0$, for all $n \in \N$, all $i \in \{1, 2\}$ and all $k \in \Z$, we have
\begin{equation}\label{Eq2}
T^k P_{\mathcal{K}_i}x_n\xi_\phi = 
u^k \cdot (P_{\mathcal{K}_i}x_n\xi_\phi)\cdot u^{-k} = P_{u^k\cdot\mathcal{K}_i\cdot u^{-k}} u^k x_n u^{-k}\xi_\phi = P_{T^k\mathcal{K}_i} u^k x_n u^{-k}\xi_\phi.
\end{equation} 
Here is a simple claim, which is just a reconstruction of Claim 1 of \cite[\S3]{Ho14}.

\begin{claim}
For any $\varepsilon>0$, there exists $k_0 \in \mathbb{N}$ such that for all $i\in \{1,2\}$, all $\xi,\eta \in \mathcal{K}_i$ and all $k \geq k_0$, we have $|\langle T^k\xi,\eta \rangle_{\rL^2(M)}| \leq \varepsilon \Vert\xi\Vert_{\rL^2(M)}\Vert\eta\Vert_{\rL^2(M)}$, that is, $T^k\mathcal{K}_i \perp_\varepsilon \mathcal{K}_i$ in the sense of \cite[Definition 2.1]{Ho12a}. 
\end{claim}

\begin{proof}[Proof of the Claim]
Denote by $J^{M_1}$ the modular conjugation on $\rL^2(M_1)$. For $\xi = \sum_{i=1}^m (e_i\xi_{\varphi_1})\otimes \xi_i, \eta = \sum_{j=1}^m (e_j\xi_{\varphi_1})\otimes \eta_j \in \mathcal{K}_1$ inside $\rL^2(M_1)\otimes (\mathfrak{H}\otimes \rL^2(M_1))$, we have
\begin{align*}
|\langle T^k\xi,\eta\rangle_{\rL^2(M)}| 
&\leq \sum_{i,j=1}^m |(u^k e_i|e_j)_{\varphi_1}|\,\Vert \xi_{i}\Vert_{\rL^2(M)}\,\Vert\eta_{j}\Vert_{\rL^2(M)} \\
&\leq 
\max_{1\leq i,j \leq m}|(u^k e_i|e_j)_{\varphi_1}| \times \Vert\xi\Vert_{\rL^2(M)}\,\Vert\eta\Vert_{\rL^2(M)}. 
\end{align*}
Similarly, for $\xi' = \sum_{i=1}^m \xi'_i\otimes (e_i\xi_{\varphi_1}), \eta' = \sum_{j=1}^m \eta'_j\otimes (e_j\xi_{\varphi_1}) \in \mathcal{K}_2$ inside $(\rL^2(M_1)\otimes \mathfrak{H})\otimes \rL^2(M_1)$, we have
\begin{align*}
|\langle T^k\xi',\eta'\rangle_{\rL^2(M)}| 
&\leq \sum_{i,j=1}^m \Vert \xi'_i\Vert_{\rL^2(M)}\,\Vert\eta'_j\Vert_{\rL^2(M)}\,|\langle u^{-k} J^{M_1}e_j\xi_{\varphi_1},J^{M_1}e_i\xi_{\varphi_1}\rangle_{\rL^2(M_1)}| \\
&\leq 
\max_{1\leq i,j \leq m}|\langle u^{-k} J^{M_1}e_j\xi_{\varphi_1},J^{M_1}e_i\xi_{\varphi_1}\rangle_{\rL^2(M_1)}| \times \Vert\xi'\Vert_{\rL^2(M)}\,\Vert\eta'\Vert_{\rL^2(M)}. 
\end{align*}

These two facts together with $\lim_{k\to\pm\infty}u^k = 0$ $\sigma$-weakly imply the desired assertion.
\end{proof}

Combining Equation \eqref{Eq2} with the parallelogram law, for all $n \in \N$, all $i \in \{1, 2\}$ and all $k \in \Z$, we have 
\begin{align*}
\Vert P_{\mathcal{K}_i}x_n\xi_\phi\Vert_{\rL^2(M)}^2 
&= \Vert T^k P_{\mathcal{K}_i}x_n\xi_\phi\Vert_{\rL^2(M)}^2 \\
&\leq 
2\Vert(u^k x_n u^{-k} - x_n)\xi_\phi\Vert_{\rL^2(M)}^2 + 2\Vert P_{T^k\mathcal{K}_i} x_n\xi_\phi\Vert_{\rL^2(M)}^2. 
\end{align*}
Thanks to this and the above Claim and since $x\in Q'\cap M^\omega$, the $\varepsilon$-orthogonality technique from \cite[Proposition 2.3]{Ho12a} works to show that $\lim_{n\to\omega}\Vert P_{\mathcal{K}_i}x_n\xi_\phi\Vert_{\rL^2(M)} = 0$ in the same way as in the proof of Claim 2 in \cite[\S3]{Ho14}. Consequently, we have $\lim_{n\to\omega}\Vert P_{\mathcal{K}_i}x_n\xi_\psi\Vert_{\rL^2(M)} \leq C\delta^{1/2}$. Since $\delta>0$ can be arbitrarily small, we finally obtain 
\begin{equation}\label{Eq3}
\lim_{n\to\omega} \Vert P_{\mathcal{K}_i}x_n\xi_\psi\Vert_{\rL^2(M)} = 0, \forall i\in \{1,2\}. 
\end{equation}

It is standard, see \cite[Theorem 3.7]{AH12}, that $\rL^2(M^\omega)$ is embedded into the ultraproduct Hilbert space $\rL^2(M)^\omega$ by $(a_n)^\omega\xi_{\varphi^\omega} \mapsto (a_n\xi_\varphi)_\omega$ for $(a_n)^\omega \in M^\omega$ with representing sequence $(a_n)_n \in \mathcal M^\omega(M)$. Remark that the other mapping $(a_n)^\omega\xi_{\psi^\omega} \mapsto (a_n\xi_\psi)_\omega$ gives exactly the same embedding since we already fix the choice (or realization) of standard forms. By \eqref{Eq3} together with \cite[Proposition 3.15, Corollary 3.27, Corollary 3.28]{AH12}, we obtain 
\begin{align*}
y(x-\rE_{M_1^\omega}(x))\xi_{\psi^\omega} 
&= 
\left(yP_\mathcal{L}x_n\xi_\psi \right)_\omega, \\
(y\rE_{M_1^\omega}(x)- J^{M^\omega} z' J^{M^\omega}\rE_{M_1^\omega}(x))\xi_{\psi^\omega} 
&= \left((y\rE_{M_1}(x_n)-J^{M} z' J^{M}\rE_{M_1}(x_n))\xi_\psi\right)_\omega, \\
J^{M^\omega} z' J^{M^\omega}(\rE_{M_1^\omega}(x)-x)\xi_{\psi^\omega} 
&= 
\left(-J^{M} z' J^{M} P_\mathcal{L}x_n\xi_\psi \right)_\omega
\end{align*}
inside $\rL^2(M)^\omega$. Note that $yP_\mathcal{L}x_n\xi_\psi$ sits in the closed linear span of $w_i WM_2^\circ \cdots M_2^\circ W\xi_\varphi$, $1 \leq i \leq \ell$, and $J^{M} z' J^{M} P_\mathcal{L}x_n\xi_\psi$ sits in the closed linear span of $W M_2^\circ \cdots M_2^\circ W w'_j\xi_\varphi$, $1 \leq j \leq \ell'$. Moreover, note that $(y\rE_{M_1}(x_n)- J^Mz' J^M \rE_{M_1}(x_n))\xi_\psi \in (y+J^Mz'J^M)\overline{M_1\xi_\psi} = (y+J^Mz'J^M)\overline{M_1\xi_\varphi}$ ({\it n.b.}~$\psi = \psi\circ \rE_{M_1}$) as well as that $J^Mz' J^M b\xi_\varphi = b \sigma_{{\rm i}/2}^\varphi(z')^*\xi_\varphi$ for every $b \in M_1$ by \cite[Lemma VIII.3.18 (ii)]{Ta03}. This shows that $(y\rE_{M_1}(x_n)- J^Mz' J^M \rE_{M_1}(x_n))\xi_\psi$ sits in the closed linear span of
$(w_i M_1 + M_1 w'_j)\xi_\varphi$, $1 \leq i \leq \ell$ and $1 \leq j \leq \ell'$.

Observe that the choice of $V$ makes the subspaces $w_i W M_2^\circ \cdots M_2^\circ W\xi_\varphi$, $W M_2^\circ \cdots M_2^\circ W w'_j\xi_\varphi$, $(w_i M_1 + M_1 w'_j)\xi_\varphi$ mutually orthogonal for all $1 \leq i \leq \ell$ and all $1 \leq j \leq \ell'$. This can easily be checked exactly in the same way as in Claim 3 of \cite[\S3]{Ho14} (which looks complicated but not difficult). Therefore, $y(x-\rE_{M_1^\omega}(x))\xi_{\psi^\omega}$, $(y\rE_{M_1^\omega}(x)-J^{M^\omega} z' J^{M^\omega}\rE_{M_1^\omega}(x))\xi_{\psi^\omega}$ and $J^{M^\omega} z' J^{M^\omega}(\rE_{M_1^\omega}(x)-x)\xi_{\psi^\omega}$ are mutually orthogonal in $\rL^2(M^\omega)$. This finishes the proof of Theorem \ref{theorem-AOP}. 
\end{proof}

\section{Proofs of Theorem \ref{thmA} and Corollary \ref{corB}}\label{proofs-main-results}

\subsection*{A key deformation/rigidity result for semifinite von Neumann algebras}

Theorem \ref{deformation/rigidity} below relies on Popa's deformation/rigidity theory \cite{Po01, Po03, Po06} and is an adaptation of Peterson's $\rL^2$-rigidity results \cite[Theorems 4.3 and 4.5]{Pe06} for {\em semifinite} von Neumann algebras using Popa's malleable deformations instead of Peterson's $\rL^2$-derivations. 

Recall from \cite{Po03, Po06} that for any inclusion $\mathcal M \subset \widetilde {\mathcal M}$ of semifinite von Neumann algebras with trace preserving conditional expectation, a trace preserving action $\R \to \Aut(\widetilde {\mathcal M}) : t \mapsto \theta_t$ is called a {\em malleable deformation} if there exists a period two trace preserving $\ast$-automorphism $\beta \in \Aut(\widetilde {\mathcal M})$ such that $\beta \circ \theta_t = \theta_{-t} \circ \beta$ for all $t \in \R$. Denote by $\rE_{\mathcal M} : \widetilde {\mathcal M} \to \mathcal M$ the unique trace preserving conditional expectation. We will simply denote by $\|\cdot\|_2$ the $\rL^2$-norm associated with the ambient faithful normal semifinite trace. By \cite[Lemma 2.1]{Po06}, any malleable deformation automatically satisfies the following {\em transversality property}:
$$\|x - \theta_{2t}(x)\|_2 \leq 2 \|\theta_t(x) - \rE_{\mathcal M}(\theta_t(x))\|_2, \forall x \in \mathcal M \cap \rL^2(\mathcal M, \Tr).$$

The main result of this subsection is the following uniform convergence theorem for malleable deformations.

\begin{thm}\label{deformation/rigidity}
Let $\mathcal B \subset \mathcal M \subset \widetilde {\mathcal M}$ be an inclusion of semifinite von Neumann algebras with trace preserving conditional expectations. Let $\R \to \Aut(\widetilde {\mathcal M}) : t \mapsto \theta_t$ be a trace preserving malleable deformation. Let $p \in \mathcal M$ be any nonzero finite trace projection and $\mathcal Q \subset p \mathcal M p$ any von Neumann subalgebra. Assume that the following conditions hold:
\begin{itemize}
\item[(i)] The $p\mathcal Mp$-$p\mathcal Mp$-bimodule $\rL^2(p\widetilde {\mathcal M}p) \ominus \rL^2(p\mathcal Mp)$ is weakly contained in the coarse $p\mathcal Mp$-$p\mathcal Mp$-bimodule $\rL^2(p\mathcal Mp) \otimes \rL^2(p\mathcal Mp)$.
\item[(ii)] The von Neumann algebra $\mathcal Q$ has no amenable direct summand.
\item[(iii)] There exists a nonprincipal ultrafilter $\omega \in \beta(\N) \setminus \N$ such that $\mathcal Q' \cap (p \mathcal M p)^\omega \npreceq_{\mathcal M^\omega} \mathcal B^\omega$.
\item[(iv)] Denote by $\rE_{\mathcal B} : \mathcal M \to \mathcal B$ the unique trace preserving conditional expectation. For every net $(v_i)_{i \in I}$ of unitaries in $\mathcal U(p \mathcal M p)$ satisfying $\lim_i \|\rE_{\mathcal B}(b^* v_i a)\|_2 = 0$ for all $a, b \in p\mathcal M$, we have $\lim_i \|\rE_{\mathcal M}(d^* v_i c)\|_2 = 0$ for all $c, d \in p(\widetilde {\mathcal M} \ominus \mathcal M)$.
\end{itemize}
Then the map $\R \to \Aut(\widetilde {\mathcal M}) : t \mapsto \theta_t$ converges uniformly on $\Ball(\mathcal Q)$ in $\|\cdot\|_2$ as $t \to 0$.
\end{thm}

\begin{proof}
Put $\mathcal P = \mathcal Q' \cap (p \mathcal M p)^\omega$. For every $t \in \R$, define $\theta_t^\omega \in \Aut(\widetilde {\mathcal M}^\omega)$ by $\theta_t^\omega((x_n)^\omega) = (\theta_t(x_n))^\omega$. We note that the map $\R \to \Aut(\widetilde {\mathcal M}^\omega) : t \mapsto \theta_t^\omega$ need not be continuous. However, exploiting Popa's spectral gap argument \cite{Po06}, we can show the following uniform convergence result.

\begin{claim}\label{claim2-convergence-ultraproduct}
The map $\R \to \Aut(\widetilde {\mathcal M}^\omega) : t \mapsto \theta_t^\omega$ converges uniformly  on $\Ball(\mathcal P)$ in $\|\cdot\|_2$ as $t \to 0$.
\end{claim}

\begin{proof}[Proof of Claim]
For the Claim, we will only use Conditions (i),(ii). Assume by contradiction that the map $\R \to \Aut(\widetilde {\mathcal M}^\omega) : t \mapsto \theta_t^\omega$ does not converge uniformly on $\Ball(\mathcal P)$ in $\|\cdot\|_2$ as $t \to 0$. Thus there exist $c > 0$, a sequence $(t_k)_k$ of positive reals such that $\lim_k t_k = 0$ and a sequence $(X_k)_k$ in $\Ball(\mathcal P)$ such that $\|X_k - \theta_{2t_k}^\omega(X_k)\|_2 \geq 2c$ for all $k \in \N$. Write $X_k = (x_{n}^{(k)})^\omega$ with $x_{n}^{(k)} \in \Ball(p \mathcal M p)$ satisfying $\lim_{n \to \omega} \|y x_n^{(k)} - x_{n}^{(k)} y\|_2 = 0$ and $2c \leq \|X_k - \theta_{2t_k}^\omega(X_k)\|_2 = \lim_{n \to \omega} \|x_n^{(k)} - \theta_{2t_k}(x_n^{(k)})\|_2$ for all $k \in \N$ and all $y \in \mathcal Q$.

Denote by $I$ the directed set of all pairs $(\mathcal F, \varepsilon)$ with $\mathcal F \subset \Ball(\mathcal Q)$ finite subset and $\varepsilon  > 0$. Let $i = (\mathcal F, \varepsilon) \in I$ and put $\delta = \min(\frac{\varepsilon}{6}, \frac{c}{8})$. Choose $k \in \N$ large enough so that $\|p - \theta_{t_k}(p)\|_2 \leq \delta$ and $\|a - \theta_{t_k}(a)\|_2 \leq \varepsilon/6$ for all $a \in \mathcal F$. Then choose $n \in \N$ large enough so that $\|x_n^{(k)} - \theta_{2t_k}(x_n^{(k)})\|_2 \geq c$ and $\|a x_n^{(k)} - x_n^{(k)} a\|_2 \leq \varepsilon /3$ for all $a \in \mathcal F$.
  
Put $\xi_i = \theta_{t_k}(x_n^{(k)}) - \rE_{ \mathcal M }(\theta_{t_k}(x_n^{(k)})) \in \rL^2( \widetilde{\mathcal M} ) \ominus \rL^2( \mathcal M )$ and $\eta_i = p \xi_i p \in \rL^2(p \widetilde{\mathcal M} p) \ominus \rL^2(p \mathcal M p)$. By the transversality property of the malleable deformation $(\theta_t)$, we have 
$$
    \|\xi_i\|_2 \geq \frac{1}{2} \|x_n^{(k)} - \theta_{2t_k} (x_n^{(k)})\|_2 \geq \frac{c}{2}.
$$
Observe that $\|p\theta_{t_k}(x_n^{(k)})p - \theta_{t_k}(x_n^{(k)})\|_2 \leq 2 \|p - \theta_{t_k}(p)\|_2 \leq 2 \delta$. Since $p \in \mathcal M$, by Pythagoras theorem, we moreover have  
$$\|p\theta_{t_k}(x_n^{(k)})p - \theta_{t_k}(x_n^{(k)})\|_2^2 = \|\rE_{ \mathcal M }(p\theta_{t_k}(x_n^{(k)})p - \theta_{t_k}(x_n^{(k)}))\|_2^2 + \|\eta_i - \xi_i\|_2^2$$ 
and hence $\|\eta_i - \xi_i\|_2 \leq 2\delta$. This implies that 
$$\|\eta_i\|_2 \geq \|\xi_i\|_2 - \|\eta_i - \xi_i\|_2 \geq  \frac{c}{2} - 2 \delta \geq \frac{c}{4}.$$
For all $x \in p \mathcal M p$, we have  
$$\|x \eta_i\|_2 = \|(1 - \rE_{ \mathcal M}) (x \theta_{t_k} (x_n^{(k)})p) \|_2 \leq \|x \theta_{t_k}(x_n^{(k)})p\|_2 \leq \|x\|_2.$$
 By Popa's spectral gap argument \cite{Po06}, for all $a \in \mathcal F \subset \Ball(p \mathcal M p)$, we have  
  \begin{align*}
    \| a \eta_i - \eta_i a \|_2
    & =
    \|(1 - \rE_{\mathcal M}) (a \theta_{t_k}( x_n^{(k)} )p - p\theta_{t_k}( x_n^{(k)}) a)\|_2 \\
    & \leq \|a \theta_{t_k}( x_n^{(k)})p - p\theta_{t_k}( x_n^{(k)}) a\|_2 \\
    & \leq
    2 \|a - \theta_{t_k}(a)\|_2 + 2 \|p - \theta_{t_k}(p)\|_2 + \|a x_n^{(k)} - x_n^{(k)} a\|_2 \\
    & \leq \frac{\varepsilon}{3} + \frac{\varepsilon}{3} + \frac{\varepsilon}{3}  = \varepsilon.
  \end{align*}
Hence $\eta_i \in \rL^2(p \widetilde {\mathcal M} p) \ominus \rL^2(p \mathcal M p)$ is a net of vectors satisfying $\limsup_i \|x \eta_i\|_2 \leq \|x\|_2$ for all $x \in p \mathcal M p$, $\liminf_i \|\eta_i\|_2 \geq~\frac{c}{4}$ and $\lim_i \|a \eta_i - \eta_i a\|_2 = 0$ for all $a \in \mathcal Q$. By Condition (i), it follows that $\mathcal Q \subset p \mathcal M p$ has an amenable direct summand by Connes's characterization of amenability \cite{Co75} for finite von Neumann algebras (see also \cite[Lemma 2.3]{Io12}). This is a contradiction to Condition (ii) and finishes the proof of the Claim.
\end{proof}

Next, we use an idea due to Peterson \cite{Pe06} in combination with the above Claim to bring down the uniform convergence to $\Ball(\mathcal Q)$.  In what follows, we will use Conditions (iii),(iv). Let $\varepsilon > 0$. By the above Claim, there exists $t_0 > 0$ such that $\|v - \theta_t^\omega(v)\|_2 < \frac{\varepsilon^2}{16}$ for all $v \in \mathcal U(\mathcal P)$ (recall $\mathcal{P} = \mathcal{Q}'\cap (p\mathcal{M}p)^\omega$) and all $t \in [-t_0, t_0]$. Fix $x \in \Ball(\mathcal Q)$ and $t \in [-t_0, t_0]$. We will show that $\|x - \theta_{2t}(x)\|_2 \leq \varepsilon$.

Denote by $I$ the directed set of all pairs $(\mathcal F, \delta)$ with $\mathcal F \subset \Ball(p \mathcal M)$ finite subset and $\delta > 0$. Fix $i = (\mathcal F, \delta) \in I$. By Condition (iii), we have that $\mathcal P \npreceq_{\mathcal M^\omega} \mathcal B^\omega$. This implies, in particular, that there exists a unitary $u \in \mathcal U(\mathcal P)$ such that $\|\rE_{\mathcal B^\omega}(b^* u a)\|_2 < \delta$ for all $a, b \in \mathcal F$. Since $p \mathcal M p$ is a finite von Neumann algebra, we may write $u = (u_n)^\omega \in \mathcal U(\mathcal P)$ for some $(u_n)_n \in \ell^\infty(\N, p \mathcal M p)$ such that $u_n \in \mathcal U(p \mathcal M p)$ for all $n \in \N$. Observe that $\lim_{n \to \omega} \|u_n x - x u_n\|_2 = \|u x - x u\|_2 = 0$, $\|\rE_{\mathcal B^\omega}(b^* u a)\|_2 = \lim_{n \to \omega} \|\rE_{\mathcal B}(b^* u_n a)\|_2 < \delta$ for all $a, b \in \mathcal F$ and $\|u - \theta_t^\omega(u)\|_2 = \lim_{n \to \omega} \|u_n - \theta_t(u_n)\|_2$. Thus, there exists $n \in \N$ large enough such that $v_i := u_n \in \mathcal U(p\mathcal M p)$ satisfies the following properties:
\begin{itemize}
\item $\|v_i x - x v_i\|_2 \leq \delta$,
\item $\|\rE_{\mathcal B}(b^* v_i a)\|_2 \leq \delta$ for all $a, b \in \mathcal F$ and
\item $\|v_i - \theta_t(v_i)\|_2 \leq \|u - \theta_t^\omega(u)\|_2 + \frac{\varepsilon^2}{16} \leq \frac{\varepsilon^2}{8}$.
\end{itemize}

Put $\delta_t(y) = \theta_t(y) - \rE_{\mathcal M}(\theta_t(y)) \in \widetilde {\mathcal M} \ominus \mathcal M$ for all $y \in p\mathcal Mp$. For all $i \in I$, we have
\begin{align}\label{mixing-delta}
\|\delta_t(x)\|_2^2 = \langle \delta_t(x), \delta_t(x)\rangle &\leq | \langle \delta_t(v_i x v_i^*), \delta_t(x)\rangle | + \|v_i x v_i^* - x\|_2 \\ \nonumber
& \leq | \langle v_i \delta_t(x) v_i^*, \delta_t(x)\rangle | + \|v_i x v_i^* - x\|_2 + 2 \|v_i - \theta_t(v_i)\|_2 \\ \nonumber
& \leq | \langle v_i \delta_t(x) v_i^*, \delta_t(x)\rangle | +  \|v_i x v_i^* - x\|_2 + \frac{\varepsilon^2}{4}.
\end{align}
Since $\lim_i \|\rE_{\mathcal B}(b^* v_i a)\|_2 = 0$ for all $a, b \in p \mathcal M$, we have $\lim_i \|\rE_{\mathcal M}(d^* v_i c)\|_2 = 0$ for all $c, d \in p (\widetilde{\mathcal M} \ominus \mathcal M)$ by Condition (iv). In particular, using Cauchy-Schwarz inequality in $\rL^2(\widetilde{\mathcal M})$, we have 
\begin{align}\label{Cauchy-Schwarz}
\limsup_i | \langle v_i \delta_t(x) v_i^*, \delta_t(x)\rangle | &= \limsup_i | \langle \delta_t(x)^* v_i \delta_t(x), v_i\rangle |  \\ \nonumber
&= \limsup_i |\langle \rE_{ \mathcal M}(\delta_t(x)^* \, v_i \, \delta_t(x)), v_i\rangle |  \\ \nonumber
&\leq \limsup_i \| \rE_{ \mathcal M }((p\delta_t(x))^* \, v_i \, p\delta_t(x)) \|_2 \, \| v_i \|_2  \\ \nonumber
& = 0.
\end{align} 
Combining \eqref{mixing-delta} and \eqref{Cauchy-Schwarz} with the first property of the net $(v_i)_{i \in I}$ and the transversality property of the malleable deformation $(\theta_t)$, we obtain
$$\|x - \theta_{2t}(x)\|_2 \leq 2 \|\delta_t(x)\|_2 \leq \varepsilon.$$
Since the above inequality holds for all $x \in \Ball(\mathcal Q)$ and all $t \in [-t_0, t_0]$, we have obtained that the map $\R \to \Aut(\widetilde {\mathcal M}) : t \mapsto \theta_t$ converges uniformly on $\Ball(\mathcal Q)$ in $\|\cdot\|_2$ as $t \to 0$. This finishes the proof of Theorem \ref{deformation/rigidity}.
\end{proof}

As a corollary to Theorem \ref{deformation/rigidity}, we obtain the following `location' result for subalgebras in semifinite amalgamated free product von Neumann algebras. For each $i \in \{1, 2\}$, let $\mathcal B \subset \mathcal M_i$ be an inclusion of $\sigma$-finite semifinite von Neumann algebras with expectation $\rE_i : \mathcal M_i \to \mathcal B$. Let $\Tr_{\mathcal B}$ be a faithful normal semifinite trace such that the weight $\Tr_{\mathcal B} \circ \rE_i$ is tracial on $\mathcal M_i$ for all $i \in \{1, 2\}$. Then the amalgamated free product $(\mathcal M, \rE) = (\mathcal M_1, \rE_1) \ast_{\mathcal B} (\mathcal M_1, \rE_1)$ is semifinite and the weight $\Tr = \Tr_{\mathcal B} \circ \rE$ is tracial on $\mathcal M$ as remarked in Section \ref{preliminaries}.

\begin{cor}\label{deformation/rigidity-AFP}
Keep the same notation as above. Assume moreover that $\mathcal B$ is amenable. Let $p \in \mathcal M$ be any nonzero finite trace projection and $\mathcal Q \subset p \mathcal M p$ any von Neumann subalgebra with no amenable direct summand such that $\mathcal Q' \cap (p \mathcal Mp)^\omega \npreceq_{\mathcal M^\omega} \mathcal B^\omega$ for some nonprincipal ultrafilter $\omega \in \beta(\N) \setminus \N$.

Then for every nonzero projection $z \in \mathcal Q' \cap p \mathcal M p$, there exists $i \in \{1, 2\}$ such that $\mathcal Qz \preceq_{\mathcal M} \mathcal M_i$.
\end{cor}

\begin{proof}
Put
$$\widetilde {\mathcal M} = \mathcal M \ast_{\mathcal B} (\mathcal B \ovt \rL(\F_2))$$ 
and consider the trace preserving free malleable deformation $(\theta_t)$ from \cite[Section 2]{IPP05} on $\widetilde {\mathcal M}$ (see \cite[Section 3]{BHR12} for further details). 

We now check that we can apply Theorem \ref{deformation/rigidity} to our situation.
\begin{itemize}
\item[(i)] Since $\mathcal B$ is amenable, the $p\mathcal Mp$-$p\mathcal Mp$-bimodule $\rL^2(p\widetilde {\mathcal M}p) \ominus \rL^2(p\mathcal Mp)$ is weakly contained in the coarse $p\mathcal Mp$-$p\mathcal Mp$-bimodule $\rL^2(p\mathcal Mp) \otimes \rL^2(p\mathcal Mp)$ (see e.g.\ the proof of \cite[Proposition 3.1]{CH08}).
\item[(ii)] By assumption, the von Neumann algebra $\mathcal Q$ has no amenable direct summand.
\item[(iii)] By assumption, we have $\mathcal Q' \cap (p \mathcal Mp)^\omega \npreceq_{\mathcal M^\omega} \mathcal B^\omega$ for some nonprincipal ultrafilter $\omega \in \beta(\N) \setminus \N$.
\item[(iv)] Let $(v_i)_{i \in I}$ be any net of unitaries in $\mathcal U(p \mathcal M p)$ such that $\lim_i \|\rE_{\mathcal B}(b^* v_i a)\|_2 = 0$ for all $a, b \in p \mathcal M$. Since $\widetilde {\mathcal M} = \mathcal M \ast_{\mathcal B} (\mathcal B \ovt \rL(\F_2))$, the proof of \cite[Theorem 2.5, Claim]{BHR12} implies that $\lim_i \|\rE_{\mathcal M}(d^* v_i c)\|_2 = 0$ for all $c, d \in p(\widetilde {\mathcal M} \ominus \mathcal M)$.
\end{itemize}
Therefore, Theorem \ref{deformation/rigidity} implies that the map $\R \to \Aut(\widetilde{\mathcal M}) : t \mapsto \theta_t$ converges uniformly on $\Ball(\mathcal Q)$ in $\|\cdot\|_2$ as $t \to 0$. Fix now any nonzero projection $z \in \mathcal Q' \cap p\mathcal Mp$. We still have that the map $\R \to \Aut(\widetilde{\mathcal M}) : t \mapsto \theta_t$ converges uniformly on $\Ball(\mathcal Q z)$ in $\|\cdot\|_2$ as $t \to 0$. Then, \cite[Theorem 3.3]{BHR12} implies that there exists $i \in \{1, 2\}$ such that $\mathcal Qz \preceq_{\mathcal M} \mathcal M_i$.
\end{proof}

\subsection*{Proof of Theorem \ref{thmA}}

Theorem \ref{thmA} will be a consequence of the following optimal result that generalizes \cite[Theorem D]{Ho14} to arbitrary free product von Neumann algebras.

\begin{thm}\label{theorem-general}
For each $i \in \{1, 2\}$, let $(M_i, \varphi_i)$ be any $\sigma$-finite von Neumann algebra endowed with a faithful normal state. Denote by $(M, \varphi) = (M_1, \varphi_1) \ast (M_2, \varphi_2)$ the free product. Let $Q \subset M$ be any von Neumann subalgebra with separable predual and with expectation such that $Q \cap M_1$ is diffuse and with expectation. Let $\omega \in \beta(\N) \setminus \N$ be any nonprincipal ultrafilter on $\N$.

Denote by $z \in \mathcal Z(Q' \cap M^\omega)$ the unique central projection such that $(Q' \cap M^\omega)z$ is diffuse and $(Q' \cap M^\omega)z^\perp$ is atomic. Then the following conditions hold:
\begin{itemize}
\item $z \in \mathcal Z(Q' \cap M) = \mathcal Z( Q' \cap M_1)$,
\item $Qz \subset z M_1 z$ and
\item $(Q' \cap M^\omega)z^\perp = (Q' \cap M)z^\perp = (Q' \cap M_1)z^\perp$.
\end{itemize}
\end{thm}

Throughout the rest of this section, let $(M, \varphi) = (M_1, \varphi_1) \ast (M_2, \varphi_2)$ be as in Theorem \ref{theorem-general}. Observe that $M_1$ is diffuse by assumption. Proposition \ref{proposition-control} (1) implies that $(M_1)' \cap M \subset M_1$. Therefore, there exists a unique faithful normal conditional expectation $\rE_{M_1} : M \to M_1$ by \cite[Th\'eor\`eme 1.5.5]{Co72}. We fix a nonprincipal ultrafilter $\omega \in \beta(\N) \setminus \N$. 

For Lemmas \ref{lemma-amenable-case} and \ref{lemma-nonamenable-case} below, we moreover fix a faithful state $\psi \in M_\ast$ such that $\psi = \psi \circ \rE_{M_1}$. Whenever $q \in M^\psi$ is a nonzero projection, put $\psi_q = \frac{\psi(q\,\cdot\,q)}{\psi(q)} \in (qMq)_\ast$. 

\begin{lem}\label{lemma-amenable-case}
Let $q \in (M_1)^{\psi}$ be any nonzero projection and $Q \subset qMq$ any non type ${\rm I}$ subfactor with separable predual that is amenable and globally invariant under the modular automorphism group $\sigma^{\psi_q}$ and such that $Q \cap qM_1q$ is diffuse. Then $Q \subset qM_1q$.
\end{lem}

\begin{proof}
The proof of Lemma \ref{lemma-amenable-case} is inspired by the one of \cite[Theorem 8.1]{Ho12b}. We will consider successively the cases when $Q$ is of type ${\rm II_1}$, of type ${\rm II_\infty}$ and of type ${\rm III}$. 

\begin{caseII_1}
Assume that $q \in (M_1)^{\psi}$ is any nonzero projection and $Q \subset qMq$ is any type ${\rm II_1}$ subfactor with separable predual that is amenable and globally invariant under the modular automorphism group $\sigma^{\psi_q}$ and such that $Q \cap qM_1q$ is diffuse. Then we have $Q \subset qM_1q$.
\end{caseII_1}

We start by showing the following claim.

\begin{claim}\label{claim1-intertwining1}
For any nonzero projection $z \in \mathcal Z(Q' \cap qMq)$, we have $Qz \preceq_M M_1$.
\end{claim}

\begin{proof}[Proof of the Claim]
By contradiction, assume that there exists a nonzero projection $z \in \mathcal Z(Q' \cap qMq)$ such that $Qz \npreceq_M M_1$.  Since $Q' \cap qMq \subset q M_1 q$ by Proposition \ref{proposition-control} (1) and $Q' \cap qMq \subset q M_1 q$ is globally invariant under the modular automorphism group $\sigma^{\psi_q}$, we have $z \in (M_1)^{\psi}$. Write $Q = \bigvee_{n \in \N} Q_n$ where $(Q_n)_n$ is an increasing sequence of finite dimensional subfactors of $Q$ of the form $Q_n \cong \mathbf M_{2^n}(\C)$. Since the inclusion$$(Q_n' \cap Q)z \subset Qz \quad \cong \quad Q_n' \cap Q \subset Q$$ 
({\it n.b.}~$Q$ is a factor) has finite index, Lemma \ref{lemma-finite-index} implies that $(Q_n' \cap Q) z \npreceq_M M_1$ for all $n \in \N$. 

Then for every $n \in \N$, choose a unitary $u_n \in \mathcal U((Q_n' \cap Q)z)$ such that $\|\rE_{M_1}(u_n)\|_\psi \leq \frac{1}{n + 1}$. Since $Qz$ is finite with expectation, 
we have $(u_n)_n \in \mathcal M^\omega(zMz)$ and hence we may define $u = (u_n)^\omega \in (zMz)^\omega = z M^\omega z \subset M^\omega$. We then have $u \in (Qz)' \cap (Qz)^\omega$ and $\rE_{M_1^\omega}(u)  = 0$ since
$$\|\rE_{M_1^\omega}(u)\|_{\psi^\omega} = \lim_{n \to \omega} \|\rE_{M_1}(u_n )\|_\psi = 0.$$

Observe that $(Qz \cap zM_1z) \oplus z^\perp M_1 z^\perp \subset M_1$ is a diffuse von Neumann subalgebra that is globally invariant under the modular automorphism group $\sigma^{\psi}$. Since $u \in (Qz)' \cap (Qz)^\omega$, we have $u \in ((Qz \cap zM_1z) \oplus z^\perp M_1 z^\perp)' \cap M^\omega$. For all $n \in \N$, since we moreover have $u \, u_n = u_n \, u$ and $u^*u = z$, Theorem \ref{theorem-AOP} implies that 
\begin{align*}
\|\rE_{M_1}(u_n) \, u - u \, \rE_{M_1}(u_n)\|_{\psi^\omega}  
&= \|(\rE_{M_1}(u_n)-u_n)u + u(u_n - \rE_{M_1}(u_n)) \|_{\psi^\omega} \\
&\geq \|u \, (u_n  - \rE_{M_1}(u_n ))\|_{\psi^\omega} \quad (\text{use Theorem \ref{theorem-AOP} here}) \\
&\geq \| z \|_\psi - \|\rE_{M_1}(u_n)\|_\psi .
\end{align*}

Observe that $\lim_{n\to\infty}\rE_{M_1}(u_n) = 0$ $\sigma$-strongly. By taking the limit as $n \to \infty$ in the above inequality, we obtain $z = 0$, a contradiction.
This finishes the proof of the Claim.
\end{proof}

The set $\mathfrak{R}$ of projections $r \in Q' \cap qMq = Q' \cap qM_1q$ (by Proposition \ref{proposition-control} (1)) such that $Qr \subset rM_1r$ attains its maximum in a unique projection $z$ that belongs to $\mathcal Z(Q' \cap qMq) = \mathcal Z(Q' \cap qM_1q)$. (In fact, $\mathfrak{R}$ is invariant under the adjoint action of $\mathcal U(Q' \cap q M_1 q)$, and $z := \bigvee_{r \in \mathfrak{R}} r \in \mathcal Z(Q' \cap qM_1q)$ must satisfy $xz = \rE_{M_1}(x)z = z \rE_{M_1}(x)z$ for all $x \in Q$.) Assume by contradiction that $z \neq q$. Put $z^\perp := q - z \in \mathcal Z(Q' \cap qMq)$. By assumption, we have $z^\perp \neq 0$.

By the previous Claim, we have that $Qz^\perp \preceq_M M_1$. Then there exist $n \geq 1$, a projection $p \in \mathbf M_n(M_1)$, a nonzero partial isometry $v \in (z^\perp M \otimes \mathbf M_{1, n}(\C))p$ and a unital normal $\ast$-homomorphism $\pi : Qz^\perp \to p\mathbf M_n(M_1)p$ such that the inclusion $\pi(Qz^\perp) \subset p\mathbf M_n(M_1)p$ is with expectation (see Theorem \ref{theorem-intertwining} due to the first named author and Isono \cite{HI15} for this important property) and $a v = v \pi(a)$ for all $a \in Qz^\perp$. By Proposition \ref{proposition-control} (1), we obtain that $v \in (z^\perp  M_1 \otimes \mathbf M_{1, n}(\C))p$ and hence $vv^* \in z^\perp(Q' \cap qMq) z^\perp = z^\perp(Q' \cap qM_1q) z^\perp$ and $Qz^\perp \, vv^* \subset vv^* \, z^\perp M_1 z^\perp \, vv^*$. Since $vv^* \leq z^\perp$, $vv^* \neq 0$ and $Q(z+ vv^*) \subset (z + vv^*) M_1 (z + vv^*)$, this contradicts the maximality of $z \in Q' \cap qM_1q$ and finishes the proof in the case when $Q$ is of type~${\rm II_1}$. 

\begin{caseII_infty}
Assume that $q \in (M_1)^{\psi}$ is any nonzero projection and $Q \subset qMq$ is any type ${\rm II_\infty}$ subfactor with separable predual that is amenable and globally invariant under the modular automorphism group $\sigma^{\psi_q}$ and such that $Q \cap qM_1q$ is diffuse. Then we have $Q \subset qM_1q$.
\end{caseII_infty}

Choose a faithful normal semifinite trace $\Tr$ on $Q$ and write $\psi_q = \Tr(T\, \cdot\,)$ for some positive nonsingular operator $T \in \rL^1(Q, \Tr)_+$ (see e.g.\ \cite[Corollary VIII.3.6, Lemma IX.2.12]{Ta03}). Define the abelian von Neumann subalgebra $B = \{T^{{\rm i}t} : t \in \R\}\dpr \subset Q$. Since $\sigma_t^{\psi_q} = \Ad(T^{{\rm i}t})$ for all $t \in \R$, we have $Q^{\psi_q} = B' \cap Q$. Observe that since the inclusion $Q \cap qM_1q \subset Q$ is globally invariant under the modular automorphism group $\sigma^{\psi_q}$, the diffuse von Neumann subalgebra $Q \cap qM_1q \subset Q$ is also semifinite and hence its centralizer $(Q \cap qM_1q)^{\psi_q}$ is diffuse (see e.g.\ \cite[Lemma 11]{Ue98b}). By Proposition \ref{proposition-control} (1) and since $B$ is abelian, we have 
$$B \subset (Q^{\psi_q})' \cap Q^{\psi_q} \subset ((Q \cap qM_1q)^{\psi_q})' \cap Q^{\psi_q} \subset Q^{\psi_q} \cap qM_1q = (Q \cap qM_1q)^{\psi_q}.$$

For every $k \in \N$, we denote by $q_k$ the spectral projection of $T$ for the interval $[\frac{1}{k},+\infty)$. Then all $q_k$  are $\Tr$-finite projections in $B$ such that $q_k \to q$, the unit of $Q$, $\sigma$-strongly as $k \to \infty$. Since $q_k \in (Q \cap qM_1q)^{\psi_q}$, the type ${\rm II_1}$ subfactor $q_k Q q_k \subset q_k M q_k$ is amenable and globally invariant under the modular automorphism group $\sigma^{\psi_{q_k}}$ and $q_k Q q_k \cap q_k M_1 q_k = q_k(Q \cap qM_1q)q_k$ is diffuse. We may then apply the result obtained in the first case to the ${\rm II_1}$ subfactor $q_k Q q_k \subset q_k M q_k$ and we have that $q_k Q q_k \subset q_k M_1 q_k$ for all $k \in \N$. Since $q_k \to q$ $\sigma$-strongly as $k \to \infty$, we obtain $Q \subset qM_1q$. This finishes the proof in the case when $Q$ is of type ${\rm II_\infty}$.

\begin{caseIII}
Assume that $q \in (M_1)^{\psi}$ is any nonzero projection and $Q \subset qMq$ is any type ${\rm III}$ subfactor with separable predual that is amenable and globally invariant under the modular automorphism group $\sigma^{\psi_q}$ and such that $Q \cap qM_1q$ is diffuse. Then we have $Q \subset qM_1q$.
\end{caseIII}

By combining results on the classification theory of amenable factors \cite{Co72, Co75, Ha85} together with \cite{FM75, Kr75}, there exists a hyperfinite ergodic nonsingular equivalence relation $\mathcal R$ defined on a standard probability space $(X, \mu)$ such that  $Q = \rL(\mathcal R)$. Put $A = \rL^\infty(X)$ and denote by $\rE_A : Q \to A$ the unique faithful normal conditional expectation. Denote by $\rE_Q : qMq \to Q$ the unique $\psi_q$-preserving conditional expectation. Choose any faithful state $\tau_A \in A_\ast$ and put $\phi = \tau_A \circ \rE_A \circ \rE_Q \in (qMq)_\ast$. Observe that $A \subset (qMq)^{\phi}$ and $Q$ is globally invariant under the modular automorphism group $\sigma^\phi$.

Let $(\mathcal R_n)_{n \in \N}$ be an increasing sequence of finite subequivalence relations of $\mathcal R$ such that $\mathcal R = \bigvee_{n \in \N} \mathcal R_n$. Put $Q_n = \rL(\mathcal R_n)$ for all $n \in \N$ so that $Q = \bigvee_{n \in \N} Q_n$. Note that $A \subset Q_n$ is still a Cartan subalgebra and $Q_n$ is globally invariant under the modular automorphism group $\sigma^\phi$ for all $n \in \N$ because $\phi |_Q = \tau_A \circ \rE_A$. Observe that since $\mathcal R_n$ is finite, that is, $\mathcal R_n$ has finite orbits almost everywhere, $Q_n$ is a countable direct sum of finite type ${\rm I}$ von Neumann algebras. Therefore using \cite[Corollary 3.19]{Ka82}, up to conjugating by a unitary in $\mathcal U(Q_n)$, the inclusion $A \subset Q_n$ is of the following form:
\begin{equation}\label{equation-inclusion}
(A \subset Q_n) \cong \left( \bigoplus_{k \in \N} \mathcal Z_n^{(k)} \otimes \C^{\oplus k} \subset \bigoplus_{k \in \N} \mathcal Z_n^{(k)} \otimes \mathbf M_k(\C) \right),
\end{equation}
where $\mathcal Z_n^{(k)}$ is a diffuse abelian von Neumann algebra for all $n, k \in \N$.

\begin{claim}
For any nonzero projection $z \in \mathcal Z(Q' \cap qMq)$, we have $Az \preceq_M M_1$.
\end{claim}

\begin{proof}[Proof of the Claim]
By contradiction, assume that there exists a nonzero projection $z \in \mathcal Z(Q' \cap qMq)$ such that $Az \npreceq_M M_1$. Observe that $z \in (qMq)^\phi \cap (M_1)^{\psi}$. Using the structure of the inclusion $A \subset Q_n$ as in \eqref{equation-inclusion}, we see that the inclusion $Q_n' \cap A \subset A$ is of the form:
\begin{equation}\label{equation-inclusion-commutant}
(Q_n' \cap A \subset A) \cong \left( \bigoplus_{k \in \N} \mathcal Z_n^{(k)} \otimes \C1 \subset \bigoplus_{k \in \N} \mathcal Z_n^{(k)} \otimes \C^{\oplus k} \right).
\end{equation}
Using \eqref{equation-inclusion-commutant}, it follows that the inclusion 
$$(Q_n' \cap A)z \subset Az \quad \cong \quad Q_n' \cap A \subset A$$
has essentially finite index and Lemma \ref{lemma-finite-index} implies that $(Q_n' \cap A)z \npreceq_M M_1$ for all $n \in \N$. (Remark that this can easily be confirmed directly, since $Q_n' \cap A \subset A$ are commutative.)

Then for every $n \in \N$, choose a unitary $u_n \in \mathcal U((Q_n' \cap A)z)$ such that $\|\rE_{M_1}(u_n)\|_\psi \leq \frac{1}{n + 1}$. Since $u_n \in (zMz)^{\phi_z}$ for all $n \in \N$, 
we may define $u = (u_n)^\omega \in (zMz)^\omega = z M^\omega z \subset M^\omega$. We then have $u \in (Qz)' \cap (Az)^\omega$ and $\rE_{M_1^\omega}(u)  = 0$. 
Observe that $u \in ((Qz \cap zM_1z) \oplus z^\perp M_1 z^\perp)' \cap M^\omega$. 
For all $n \in \N$, 
Theorem \ref{theorem-AOP} implies, as in Case type $\rm II_1$, that
$$
\|\rE_{M_1}(u_n) \, u - u \, \rE_{M_1}(u_n)\|_{\psi^\omega} 
\geq \| z \|_\psi - \|\rE_{M_1}(u_n)\|_\psi
$$
and hence $z = 0$, a contradiction. This finishes the proof of the Claim.
\end{proof}

The set of projections $r \in Q' \cap qMq = Q' \cap qM_1q$ (by Proposition \ref{proposition-control} (1)) such that $Qr \subset rM_1r$ attains its maximum in a unique projection $z$ that belongs to $\mathcal Z(Q' \cap qMq) = \mathcal Z(Q' \cap qM_1q)$ (see Case type $\rm II_1$). Assume by contradiction that $z \neq q$. Put $z^\perp := q - z \in \mathcal Z(Q' \cap qMq)$. By assumption, we have $z^\perp \neq 0$ and moreover $z^\perp \in Q' \cap qMq \subset A' \cap qMq$.

By the previous Claim, we have that $Az^\perp \preceq_M M_1$. Then there exist $n \geq 1$, a projection $p \in \mathbf M_n(M_1)$, a nonzero partial isometry $v \in (z^\perp M \otimes \mathbf M_{1, n}(\C))p$ and a unital normal $\ast$-homomorphism $\pi : Az^\perp \to p\mathbf M_n(M_1)p$ such that the inclusion $\pi(Az^\perp) \subset p\mathbf M_n(M_1)p$ is with expectation (see Theorem \ref{theorem-intertwining}) and $a v = v \pi(a)$ for all $a \in Az^\perp$. Since $z^\perp \in Q' \cap qMq \subset A' \cap qMq$, we may define the unital normal $\ast$-homomorphism $\iota : A \to Az^\perp : a \mapsto az^\perp$. Then $\pi \circ \iota : A \to p\mathbf M_n(M_1)p$ is unital normal $\ast$-homomorphism such that the inclusion $(\pi \circ \iota)(A) \subset p\mathbf M_n(M_1)p$ is with expectation and $a \, v = \iota(a) \, v = v \, \pi(\iota(a)) = v \, (\pi\circ \iota)(a)$ for all $a \in A$. 

Put $N = \mathcal N_{qMq}(A)\dpr$ and observe that $Q \subset N$. Since $v^*v \in (\pi \circ \iota)(A)' \cap p\mathbf M_n(M)p$ and since $(\pi \circ \iota)(A) \subset p\mathbf M_n(M_1)p$ is diffuse and with expectation, we have $v^*v \in (\pi \circ \iota)(A)' \cap p\mathbf M_n(M_1)p$ by Proposition \ref{proposition-control} (2) and hence we may assume that $p = v^* v$. Since the inclusion $A \subset N$ is regular, we moreover have $v^* \,  N  \, v \subset  p\mathbf M_n(M_1)p$ by Proposition \ref{proposition-control} (2). 

We have $vv^* \in (Az^\perp)' \cap z^\perp M z^\perp = z^\perp(A' \cap qMq)z^\perp \subset z^\perp N z^\perp$. Since the inclusion $Q \subset N$ is with expectation (because so is $Q \subset qMq$)
and since $Q$ is of type ${\rm III}$, it follows that $N$ is also of type ${\rm III}$ (see \cite[Lemma V.2.29]{Ta02}) and so is $z^\perp N z^\perp$. If we denote by $r \in \mathcal Z(z^\perp N z^\perp)$ the central support in $z^\perp N z^\perp$ of the projection $vv^* \in z^\perp N z^\perp$, we have $vv^* \sim r$ in $z^\perp N z^\perp$. There exists a  partial isometry $u \in z^\perp N z^\perp$ such that $u^*u = vv^*$ and $uu^* = r$. We have $(uv)^* Nr \, (uv) \subset  p\mathbf M_n(M_1)p$. So, up to replacing $v$ by $uv$, we may assume that $v^* \, z^\perp N z^\perp \, v \subset  p\mathbf M_n(M_1)p$, $vv^* \in \mathcal Z(z^\perp N z^\perp)$ and $p = v^*v$ .

This implies that $ z^\perp N z^\perp \, v \subset  v \, p\mathbf M_n(M_1)p$ and hence $ Q z^\perp \, v \subset  v \, p\mathbf M_n(M_1)p$. This further implies that $(Q \cap qM_1q)z^\perp \, v \subset  v \, p\mathbf M_n(M_1)p$. Since $p = v^* v$, $vv^* \in \mathcal{Z}(z^\perp N z^\perp)$ and $Q \subset N$,
the mapping $\rho : (Q \cap qM_1q)z^\perp \to p\mathbf M_n(M_1)p : x \mapsto v^* x v$ defines a unital normal $\ast$-homomorphism such that $xv = v\rho(x)$ for all $x \in (Q \cap qM_1q)z^\perp$.  Observe that $z \in M^\psi$ and hence $(Q \cap qM_1q)z^\perp \subset z^\perp M z^\perp$ is with expectation. By Proposition \ref{proposition-control} (1), we obtain that $v \in (z^\perp  M_1 \otimes \mathbf M_{1, n}(\C))p$ and hence $vv^* \in z^\perp(Q' \cap qMq) z^\perp = z^\perp(Q' \cap qM_1q) z^\perp$ and $Qz^\perp \, vv^* \subset vv^* \, z^\perp M_1 z^\perp \, vv^*$. Since $vv^* \leq z^\perp$, $vv^* \neq 0$ and $Q(z+ vv^*) \subset (z + vv^*) M_1 (z + vv^*)$, this contradicts the maximality of $z \in Q' \cap qM_1q$ and finishes the proof in the case when $Q$ is of type ${\rm III}$.

Since we have successively treated the cases when $Q$ is of type ${\rm II_1}$, of type ${\rm II_\infty}$ and of type ${\rm III}$, this finishes the proof of Lemma \ref{lemma-amenable-case}.
\end{proof}

\begin{lem}\label{lemma-nonamenable-case}
Let $q \in (M_1)^{\psi}$ be any nonzero projection and $Q \subset qMq$ any subfactor with separable predual that is not amenable and globally invariant under the modular automorphism group $\sigma^{\psi_q}$ and such that $Q \cap qM_1q$ and $Q' \cap (qMq)^\omega$ are diffuse. Then $Q \subset qM_1q$.
\end{lem}

\begin{proof}
The proof, inspired by the one of \cite[Theorem E]{Ho12b}, relies on Connes-Takesaki's structure theory \cite{Co72, Ta03} and uses Corollary \ref{deformation/rigidity-AFP}. 

The novel aspect of the proof consists in combining \cite[Theorem 4.1]{AH12} and \cite[Theorem 2.10]{MT13} in order to obtain the following canonical inclusions of semifinite von Neumann algebras with trace preserving conditional expectations:
$$\core_{\phi}(M) \subset \core_{\phi^\omega}(M^\omega) \subset (\core_{\phi}(M))^\omega$$
with $\phi = \varphi$ or $\phi = \psi$. 
More precisely, if we denote by $\rE_\omega^\phi : (\core_{\phi}(M))^\omega \to \core_{\phi}(M)$ the canonical faithful normal conditional expectation and by $\Tr_{\phi}$ (resp.\ $\Tr_{\phi^\omega}$) the canonical faithful normal semifinite trace on $\core_{\phi}(M)$ (resp.\ $\core_{\phi^\omega}(M^\omega)$), we have that $\Tr_{\phi} \circ \rE_\omega^\phi$ is a faithful normal semifinite trace on $(\core_{\phi}(M))^\omega$ and $(\Tr_{\phi} \circ \rE_\omega^\phi) |_{\core_{\phi^\omega}(M^\omega)} = \Tr_{\phi^\omega}$. We will simply use the notation $\|\cdot\|_2$ for the $\rL^2$-norm associated with any of the faithful normal semifinite traces considered above. We will use throughout the proof the identification $\rL_{\phi}(\R) = \rL_{\phi^\omega}(\R) \subset \core_{\phi^\omega}(M^\omega)$.

Since $q \in M^\psi$ and $Q \subset qMq$ is globally invariant under the modular automorphism group $\sigma^{\psi_q}$, we may define $\core_{\psi_q}(Q) = Q \rtimes_{\sigma^{\psi_q}} \R$ and regard $\core_{\psi_q}(Q) \subset \pi_\psi(q) \core_\psi(M) \pi_\psi(q)$ naturally. Fix an arbitrary nonzero finite trace projection $r \in \rL_{\psi}(\R)$ and put $\mathcal M = \core_{\varphi}(M)$, $p = \Pi_{\varphi, \psi}(r) \in \mathcal M$, $\mathcal Q = \Pi_{\varphi, \psi}(r\core_{\psi_q}(Q)r)$ and $\mathcal{P} = \mathcal Q' \cap (p\pi_\varphi(q)\mathcal M \pi_\varphi(q)p)^\omega$. Observe that 
$$p\pi_\varphi(q) = \Pi_{\varphi, \psi}(r \pi_\psi(q)) = \Pi_{\varphi, \psi}(\pi_\psi(q) r) = \pi_\varphi(q)p$$ 
defines a nonzero projection in $\mathcal M$ and is the unit of $\mathcal Q$.

\begin{claim}
We have $\mathcal{P} \npreceq_{\mathcal M^\omega} (\rL_{\varphi}(\R))^\omega$.
\end{claim}

\begin{proof}[Proof of the Claim]
The proof uses an idea of \cite[Lemma 9.5]{Io12}. By contradiction, assume that $\mathcal{P} \preceq_{\mathcal M^\omega} (\rL_{\varphi}(\R))^\omega$. By \cite[Lemma 2.3]{BHR12}, there exist $\delta > 0$ and a finite subset $\mathcal F \subset p\pi_\varphi(q) \mathcal M^\omega$ such that 
\begin{equation}\label{equation-intertwining1}
\sum_{a, b \in \mathcal F} \|\rE_{(\rL_{\varphi}(\R))^\omega}(b^* u a)\|_2^2 > \delta, \forall u \in \mathcal U(\mathcal{P}).
\end{equation}
For each $a \in \mathcal F$, write $a = (a_n)^\omega$ with a fixed sequence $(a_n)_n \in p\pi_\varphi(q)\mathcal M^\omega(\mathcal M)$.

We next show that there exists $n \in \N$ such that 
\begin{equation}\label{equation-intertwining2}
\sum_{a, b \in \mathcal F} \|\rE_{(\rL_{\varphi}(\R))^\omega}(b_n^* u a_n)\|_2^2 \geq \delta, \forall u \in \mathcal U(\mathcal{P}).
\end{equation}
Assume by contradiction that this is not the case. Then for each $n \in \N$, there exists $u_n \in \mathcal U(\mathcal{P})$ such that 
$$\sum_{a, b \in \mathcal F} \|\rE_{(\rL_{\varphi}(\R))^\omega}(b_n^* u_n a_n)\|_2^2 < \delta.$$ 
Since $p\pi_\varphi(q)\mathcal M \pi_\varphi(q)p$ is a finite von Neumann algebra, we may write  $u_n = (u^{(n)}_m)^\omega$ for some sequence $(u^{(n)}_m)_m \in \ell^\infty(\N, p\pi_\varphi(q)\mathcal M \pi_\varphi(q)p)$ such that $u_m^{(n)} \in \mathcal U(p\pi_\varphi(q) \mathcal M \pi_\varphi(q)p)$ for all $m \in \N$. Then we have  
$$\lim_{m \to \omega} \sum_{a, b \in \mathcal F} \|\rE_{\rL_{\varphi}(\R)}(b_n^* u^{(n)}_m a_n)\|_2^2 < \delta.$$
Fix a $\|\cdot\|_2$-dense countable subset $\{y_n : n \in \N\} \subset \mathcal Q$. Since $\lim_{m \to \omega} \|y_j u_m^{(n)} - u_m^{(n)} y_j\|_2 = \|y_j u_n - u_n y_j\|_2 = 0$ for all $n \in \N$ and all $0 \leq j \leq n$, we may choose $m_n \in \N$ large enough so that $v_n := u_{m_n}^{(n)} \in \mathcal U(p\pi_\varphi(q)\mathcal M \pi_\varphi(q)p)$ satisfies $\|y_j v_n - v_n y_j\|_2 \leq \frac{1}{n + 1}$ for all $0 \leq j \leq n$ and $\sum_{a, b \in \mathcal F} \|\rE_{\rL_{\varphi_q}(\R)}(b_n^* v_n a_n)\|_2^2 \leq \delta$. Since $p\pi_\varphi(q)\mathcal M \pi_\varphi(q)p$ is finite, we may define $v := (v_n)^\omega \in (p\pi_\varphi(q) \mathcal M \pi_\varphi(q) p)^\omega$. We moreover have $v \in \mathcal U(\mathcal{P})$ and 
\begin{equation}\label{equation-intertwining3}
\sum_{a, b \in \mathcal F} \|\rE_{(\rL_{\varphi}(\R))^\omega}(b^* v a)\|_2^2 = \lim_{n \to \omega} \sum_{a, b \in \mathcal F} \|\rE_{\rL_{\varphi}(\R)}(b_n^* v_n a_n)\|_2^2 \leq \delta.
\end{equation}
Equations \eqref{equation-intertwining1} and \eqref{equation-intertwining3} give a contradiction. This shows that Equation \eqref{equation-intertwining2} holds. Therefore, up to replacing the finite subset $\mathcal F \subset p\pi_\varphi(q) \mathcal M^\omega$ by $\{a_n : a \in \mathcal F\} \subset p\pi_\varphi(q) \mathcal M$, we may assume that $\mathcal F \subset p\pi_\varphi(q) \mathcal M$ in Equation \eqref{equation-intertwining1}.

Since $Q' \cap (qMq)^\omega$ is diffuse and $Q$ is globally invariant under the modular automorphism group $\sigma^{\psi_q}$, we know that $Q' \cap ((qMq)^\omega)^{\psi_q^\omega}$ is diffuse by \cite[Theorem 2.3]{HR14}. We may then choose a sequence $(u_n)_n \in \mathcal M^\omega(qMq)$ such that $(u_n)^\omega \in \mathcal U(Q' \cap ((qMq)^\omega)^{\psi_q^\omega})$ and $\lim_{n\to\infty} u_n = 0$ $\sigma$-weakly (see the first and second paragraphs in the proof of \cite[Theorem A]{HR14}). Observe that $(p \pi_\varphi(u_n) p)_n \in \ell^\infty(\N, p\pi_\varphi(q)\mathcal M \pi_\varphi(q)p)$ and 
\begin{align*}
\pi_{\varphi^\omega}((u_n)^\omega) p &= \Pi_{\varphi^\omega, \psi^\omega}(\pi_{\psi^\omega}((u_n)^\omega) r) \\
&= \Pi_{\varphi^\omega, \psi^\omega}(r\pi_{\psi^\omega}((u_n)^\omega) r) \\
&= (\Pi_{\varphi, \psi}(r \pi_\psi(u_n) r))^\omega \\
&= (p \pi_\varphi(u_n) p)^\omega \in \mathcal U(\mathcal P).
\end{align*}
Since $\mathcal F \subset p\pi_\varphi(q) \mathcal M$, using Lemma \ref{lemma-intertwining-core} (with letting the $Q$ there be the trivial algebra), we obtain $\lim_{n \to \omega} \|\rE_{\rL_{\varphi}(\R)}(b^* \, p\pi_\varphi(u_n)p \,  a)\|_2 = 0$ for all $a, b \in \mathcal F$. This implies that 
\begin{equation}\label{equation-intertwining4}
\sum_{a, b \in \mathcal F} \|\rE_{(\rL_{\varphi}(\R))^\omega}(b^*\, \pi_{\varphi^\omega}((u_n)^\omega) p \, a)\|_2^2 = \lim_{n \to \omega} \sum_{a, b \in \mathcal F} \|\rE_{\rL_{\varphi}(\R)}(b^*\, p\pi_\varphi(u_n) p \, a)\|_2^2 = 0.
\end{equation}
Equation \eqref{equation-intertwining1} with $\mathcal F \subset p\pi_\varphi(q) \mathcal M$ and Equation \eqref{equation-intertwining4} give a contradiction. This finishes the proof of the Claim.
\end{proof}

Next, for each $i \in \{1, 2\}$, put $\mathcal M_i = \core_{\varphi}(M_i)$. We have $\mathcal M = \mathcal M_1 \ast_{\rL_\varphi(\R)} \mathcal M_2$ (see \cite[Theorem 5.1]{Ue98a}). Observe that since $M_1$ is globally invariant under $\sigma^\psi$, we have $\Pi_{\varphi, \psi}(\core_\psi(M_1)) = \core_\varphi(M_1) = \mathcal M_1$. Since $r \in \rL_\psi(\R) \subset \core_\psi(M_1)$, we have $p = \Pi_{\varphi, \psi} (r) \in \mathcal M_1$. Since $Q \cap qM_1q$ is diffuse and globally invariant under $\sigma^{\psi_q}$, we have $\Pi_{\varphi, \psi}(r \core_{\psi_q}(Q \cap qM_1q) r) \npreceq_{\mathcal M} \rL_{\varphi}(\R)$ by Lemma \ref{lemma-intertwining-core-bis}. Then \cite[Theorem 2.5]{BHR12} implies that $(\Pi_{\varphi, \psi}(r \core_{\psi_q}(Q \cap qM_1q) r))' \cap p\pi_\varphi(q)\mathcal M\pi_\varphi(q)p \subset p\pi_\varphi(q)\mathcal M_1\pi_\varphi(q)p$ and hence 
$$\mathcal Q' \cap p\pi_\varphi(q)\mathcal M\pi_\varphi(q)p = \mathcal Q' \cap p\pi_\varphi(q)\mathcal M_1\pi_\varphi(q)p.$$ 
The set of projections $s \in \mathcal Q' \cap p\pi_\varphi(q) \mathcal M \pi_\varphi(q)p = \mathcal Q' \cap p\pi_\varphi(q) \mathcal M_1 \pi_\varphi(q)p$ such that $\mathcal Qs \subset s\mathcal M_1s$ attains its maximum in a unique projection $z$ that belongs to $\mathcal Z(\mathcal Q' \cap p\pi_\varphi(q) \mathcal M \pi_\varphi(q)p) = \mathcal Z(\mathcal Q' \cap p\pi_\varphi(q) \mathcal M_1\pi_\varphi(q)p)$. Assume by contradiction that $z \neq p\pi_\varphi(q)$. Put $z^\perp := p\pi_\varphi(q) - z \in \mathcal Z(\mathcal Q' \cap p\pi_\varphi(q)\mathcal M\pi_\varphi(q)p)$. By assumption, we have $z^\perp \neq 0$.

Observe that since $Q \subset q M q$ is a subfactor that is not amenable, $\mathcal Q = \Pi_{\varphi, \psi}( r \core_{\psi_q}(Q) r)$ has no amenable direct summand by \cite[Proposition 2.8]{BHR12}. By the previous Claim, we moreover have $\mathcal{P} \npreceq_{\mathcal M^\omega} (\rL_\varphi(\R))^\omega$. Then Corollary \ref{deformation/rigidity-AFP} implies that there exists $i \in \{1, 2\}$ such that $\mathcal Q z^\perp \preceq_{\mathcal M} \mathcal M_i$.  Hence, there exist $n \geq 1$, a finite trace projection $f \in \mathbf M_n(\mathcal M_i)$ (with respect to the canonical trace $\Tr_{\varphi_i} \otimes \tr_n$), a nonzero partial isometry $v \in (z^\perp \mathcal M \otimes \mathbf M_{1, n}(\C))f$ and a unital normal $\ast$-homomorphism $\pi : \mathcal Qz^\perp \to f \mathbf M_n(\mathcal M_i) f$ such that $xv = v \pi(x)$ for all $x \in \mathcal Q z^\perp$. In particular, we have $\Pi_{\varphi, \psi}(r \core_{\psi_q}(Q \cap qM_1q) r) \, v \subset  v \,  f \mathbf M_n(\mathcal M_i) f$. Since $\Pi_{\varphi, \psi}(r \core_{\psi_q}(Q \cap qM_1q) r) \npreceq_{\mathcal M} \rL_{\varphi}(\R)$, \cite[Theorem 2.5]{BHR12} and its Claim imply that $i = 1$ and $v \in (z^\perp \mathcal M_1 \otimes \mathbf M_{1, n}(\C))f$. Therefore we have $vv^* \in \mathcal Q' \cap p\pi_\varphi(q)\mathcal M_1 \pi_\varphi(q)p$, $vv^* \neq 0$, $vv^* \leq z^\perp$ and $\mathcal Q (z + vv^*) \subset (z +vv^*) \mathcal M_1 (z + vv^*)$. This contradicts the maximality of the projection $z \in \mathcal Q' \cap p\pi_\varphi(q)\mathcal M_1\pi_\varphi(q)p$.

Thus, we have $z = p\pi_\varphi(q)$ and hence 
$$\Pi_{\varphi, \psi}(r\core_{\psi_q}(Q)r) = \mathcal Q \subset p\pi_\varphi(q)\mathcal M_1\pi_\varphi(q)p = \Pi_{\varphi, \psi}(r \core_{\psi_q}(qM_1q) r).$$ 
This implies that $r\core_{\psi_q}(Q)r \subset r \core_{\psi_q}(qM_1q) r$. Since this holds for every nonzero finite trace projection $r \in \rL_\psi(\R)$, we obtain $\core_{\psi_q}(Q) \subset \core_{\psi_q}(qM_1q)$.  Observe that $\pi_{\psi_q}(qMq) \subset \core_{\psi_q}(qMq)$ is the fixed point algebra by an action of $\R$, called the dual action of $\sigma^{\psi_q}$, (see \cite[Theorem X.2.3 (i)]{Ta03}) and hence there exists a (non-normal) conditional expectation ${\mathrm F} :  \core_{\psi_q}(qMq) \to \pi_{\psi_q}(qMq)$ such that ${\mathrm F}(\core_{\psi_q}(q M_1 q)) = \pi_{\psi_q}(q M_1 q)$. By applying the conditional expectation $\mathrm F$ to $\pi_{\psi_q}(Q) \subset \core_{\psi_q}(Q) \subset \core_{\psi_q}(qM_1q)$, we obtain $\pi_{\psi_q}(Q) \subset \pi_{\psi_q}(qM_1q)$ and hence $Q \subset qM_1q$. This finishes the proof of Lemma \ref{lemma-nonamenable-case}.
\end{proof}

\begin{proof}[Proof of Theorem \ref{theorem-general}]
Since both $Q$ and $Q \cap M_1$ are with expectation in $M$, we may choose a faithful state $\psi \in M_\ast$ such that both $Q$ and $Q \cap M_1$ are globally invariant under the modular automorphism group $\sigma^\psi$. Denote by $\R \to \mathcal U(M) : t \mapsto u_t = [D\psi : D\varphi]_t$ the Connes Radon-Nikodym cocycle (see \cite[Th\'eor\`eme 1.2.1]{Co72}) satisfying $\sigma_t^\psi = \Ad(u_t) \circ \sigma_t^\varphi$ for all $t \in \R$.

Fix any $t \in \R$. Define the unital normal $\ast$-isomorphism $\pi_t : Q \cap M_1 \to M : x \mapsto u_t{}^* x u_t$. Observe that 
$$\pi_t(Q \cap M_1) = u_t{}^* \, Q \cap M_1 \, u_t = u_t{}^* \, \sigma_t^\psi(Q \cap M_1) \, u_t = \sigma_t^\varphi(Q \cap M_1) \subset \sigma_t^\varphi(M_1) = M_1$$
and $x \, u_t = u_t \, \pi_t(x)$ for all $x \in Q \cap M_1$. Since $Q \cap M_1 \subset M_1$ is diffuse and with expectation, Proposition \ref{proposition-control} (1) implies that $u_t \in \mathcal U(M_1)$. Since this holds for every $t \in \R$, we obtain 
$$\sigma_t^\psi(M_1) = u_t \, \sigma_t^\varphi(M_1) \, u_t{}^* = u_t \, M_1 \, u_t{}^* = M_1.$$
This implies that $\psi = \psi \circ \rE_{M_1}$ where $\rE_{M_1} : M \to M_1$ is the unique faithful normal conditional expectation.

Since $Q \cap M_1 \subset M_1$ is diffuse and with expectation, we have $Q' \cap M \subset (Q \cap M_1)' \cap M = (Q \cap M_1)' \cap M_1$ by Proposition \ref{proposition-control} (1) and hence $Q' \cap M = Q' \cap M_1$. Denote by $z \in \mathcal Z(Q' \cap M^\omega)$ the unique central projection such that $(Q' \cap M^\omega)z$ is diffuse and $(Q' \cap M^\omega)z^\perp$ is atomic. By \cite[Theorem 2.3]{HR14}, we have $z \in \mathcal Z(Q' \cap M) = \mathcal Z(Q' \cap M_1)$ and $(Q' \cap M^\omega)z^\perp = (Q' \cap M)z^\perp = (Q' \cap M_1)z^\perp$. Observe that $z \in (M_1)^\psi$.

Denote by $(z_n)_n$ a sequence of central projections in $\mathcal Z(Qz)$ such that $\sum_n z_n = z$, $Q z_0$ has a diffuse center and $Qz_n$ is a diffuse factor for all $n \geq 1$. We have $\mathcal Z(Qz) \subset (Qz)' \cap zM^\psi z = z(Q' \cap M^\psi)z = z(Q' \cap (M_1)^{\psi})z$. Moreover, since $\mathcal Z(Qz) z_0 \subset z_0M_1z_0$ is diffuse and globally invariant under the modular automorphism group $\sigma^{\psi_{z_0}}$, we have $Qz_0 \subset (\mathcal Z(Qz)z_0)' \cap z_0Mz_0 = (\mathcal Z(Qz)z_0)' \cap z_0M_1z_0$ by Proposition \ref{proposition-control} (1). Finally, for all $n \geq 1$, since $Qz_n \subset z_n M z_n$ is a non type ${\rm I}$ subfactor that is globally invariant under the modular automorphism group $\sigma^{\psi_{z_n}}$ and such that $Qz_n \cap z_n M_1 z_n = (Q \cap M_1)z_n$ and $(Qz_n)' \cap (z_n M z_n)^\omega = (Q' \cap M^\omega)z_n$ are diffuse, Lemma \ref{lemma-amenable-case}, in the case when $Qz_n$ is amenable, and Lemma \ref{lemma-nonamenable-case}, in the case when $Qz_n$ is nonamenable, imply that $Qz_n \subset z_n M_1 z_n$. Therefore, we have $Qz \subset zM_1 z$. This finishes the proof of Theorem \ref{theorem-general}.
\end{proof}

We can finally deduce the main results of this paper.

\begin{proof}[Proof of Theorem \ref{thmA}]
By applying Theorem \ref{theorem-general} to the case when the projection $z \in \mathcal Z(Q' \cap M^\omega)$ satisfies $z = 1$, we obtain $Q \subset M_1$.
\end{proof}

\begin{proof}[Proof of Corollary \ref{corB}]
Since both $Q$ and $Q \cap M_1$ are with expectation and $Q \cap M_1$ is diffuse, using Lemma \ref{lemma-centralizer}, we may choose a faithful state $\psi \in M_\ast$ such that both $Q$ and $Q \cap M_1$ are globally invariant under the modular automorphism group $\sigma^\psi$ and the centralizer $(Q \cap M_1)^\psi$ is diffuse. Note that by the proof of Theorem \ref{theorem-general}, $M_1$ is also globally invariant under the modular automorphism group $\sigma^\psi$. Next, choose a diffuse abelian von Neumann subalgebra with separable predual $A \subset (Q \cap M_1)^\psi$.  

Let $x \in Q$ be any element. Denote by $Q_0 \subset M$ the von Neumann subalgebra generated by the set $\{ \sigma_t^\psi(y) : t \in \R, \, y = x \text{ or } y \in A\}$. Observe that $Q_0 \subset M$ has separable predual and is globally invariant under the modular automorphism group $\sigma^\psi$. Since $Q$ is amenable and $Q_0 \subset Q$ is with expectation, it follows that $Q_0$ is also amenable. (It is true even in the non-separable case that amenability implies injectivity. See \cite{Co76}.) Since $A \subset (Q_0 \cap M_1)^\psi$ and since $A$ is diffuse, $(Q_0 \cap M_1)^\psi$ is diffuse and so is $Q_0 \cap M_1$ (see e.g.\ \cite[Theorem IV.2.2.3]{Bl06}).

Since $Q_0$ is diffuse, amenable and with separable predual, the central sequence algebra $Q_0' \cap Q_0^\omega$ is diffuse (see e.g.\ \cite[Proposition 2.6]{Ho14}). Since $Q_0 \subset M$ is with expectation, the inclusion $Q_0' \cap Q_0^\omega \subset Q_0' \cap M^\omega$ is with expectation and hence $Q_0' \cap M^\omega$ is diffuse. Since $Q_0 \cap M_1$ is moreover diffuse and with expectation, we obtain that $Q_0 \subset M_1$ by Theorem \ref{thmA} and hence $x \in M_1$. Since this holds true for all $x \in Q$, we deduce $Q \subset M_1$.
\end{proof}

\appendix

\section{Bicentralizer problem for free product von Neumann algebras}\label{appendix} Let $(M, \varphi)$ be any $\sigma$-finite von Neumann algebra endowed with a faithful normal state. Following \cite{Ha85}, the {\em asymptotic centralizer} of $\varphi$ is defined by 
$$\AC(M, \varphi) = \left\{ (x_n)_n \in \ell^\infty(\N, M) : \lim_{n \to \infty} \|x_n \varphi - \varphi x_n\| = 0 \right\}$$
and the {\em bicentralizer} of $\varphi$ is defined by 
$$\Bic(M, \varphi) = \left\{ a \in M : \lim_{n \to \infty} \|a x_n - x_n a\|_\varphi = 0, \forall (x_n)_n \in \AC(M, \varphi) \right\}.$$
Haagerup showed in \cite{Ha85} that any amenable type ${\rm III_1}$ factor with separable predual has trivial bicentralizer. It is an open problem, known as Connes's bicentralizer problem, to decide whether any type ${\rm III_1}$ factor with separable predual has trivial bicentralizer.

It was recently showed in \cite[Proposition 3.3]{HI15} that $\Bic(M, \varphi) = ((M^\omega)^{\varphi^\omega})' \cap M$ for every nonprincipal ultrafilter $\omega \in \beta(\N) \setminus \N$. Using this characterization, we give a short proof of an unpublished result due to the second named author showing that Connes's bicentralizer problem has a positive solution for all type ${\rm III_1}$ free product factors.

For each $i \in \{1, 2\}$, let $(M_i, \varphi_i)$ be any nontrivial $\sigma$-finite von Neumann algebra endowed with a faithful normal state. Assume moreover that $\ker(\sigma^{\varphi_1})\cap\ker(\sigma^{\varphi_2}) = \{0\}$. Denote by $(M, \varphi) = (M_1, \varphi_1) \ast (M_2, \varphi_2)$ the  free product. By \cite[Theorem 4.1]{Ue10}, we have $M = M_c \oplus M_d$ where $M_c$ is a type ${\rm III_1}$ factor and $M_d = 0$ or $M_d$ is a multimatrix algebra. Put $\varphi_c = \frac{1}{\varphi(1_{M_c})} \varphi|_{M_c}$.

\begin{thm}\label{bicentralizer}
Keep the same notation as above. Then $\Bic(M_c, \varphi_c) = \C 1_{M_c}$.
\end{thm}

\begin{proof}
In the case when both $M_1$ and $M_2$ are atomic, $\varphi_c$ is an almost periodic state such that $((M_c)^{\varphi_c})' \cap M_c^\omega = \C 1_{M_c}$ by \cite[Theorem 2.2]{Ue11}. Then we have $\Bic(M_c, \varphi_c) \subset ((M_c)^{\varphi_c})' \cap M_c= \C 1_{M_c}$.

Next, we may assume that $M_1$ has a diffuse direct summand. Since $M_c$ is of type ${\rm III}$ and using \cite[Lemma 2.2]{Ue10}, up to cutting down $M$ by the central projection in $M_1$ that supports the diffuse direct summand of $M_1$, we may assume without loss of generality that $M_1$ is diffuse. In that case, we have $M = M_c$. Observe that $M_1^\omega$ and $M_2^\omega$ are both globally invariant under the modular automorphism group $\sigma^{\varphi^\omega}$ and are $\ast$-free inside $M^\omega$ with respect to the state $\varphi^\omega$ (see \cite[Proposition 4]{Ue00}). Letting $P = M_1^\omega \vee M_2^\omega$, we have $(P, \varphi^\omega|_P ) \cong (M_1^\omega, \varphi_1^\omega) \ast (M_2^\omega, \varphi_2^\omega)$ and $M \subset P \subset M^\omega$.

Since $M_1$ is diffuse and $M_2 \neq \C 1$, we have that $(M_1^\omega)^{\varphi_1^\omega}$ is diffuse and $(M_2^\omega)^{\varphi_2^\omega} \neq \C 1$ by Proposition \ref{proposition-centralizer-ultraproduct}. Using \cite[Proposition 3.3]{HI15} and Proposition \ref{proposition-control} (1), we have
$$\Bic(M, \varphi) = ((M^\omega)^{\varphi^\omega})' \cap M \subset ((M_1^\omega)^{\varphi_1^\omega})' \cap P \cap  M \subset ((M_1^\omega)^{\varphi_1^\omega})' \cap M_1^\omega \cap M \subset M_1.$$
Next, one can choose an invertible element $w \in (M_2^\omega)^{\varphi_2^\omega}$ such that $\varphi_2^\omega(w) = 0$. For all $y \in \Bic(M, \varphi) \subset M_1$ such that $\varphi_1(y) = 0$, using the freeness with respect to $\varphi^\omega$ and since $yw = wy$, we have 
$$
\varphi^\omega(w^* y^*y w) = \varphi^\omega(w^* y^* w y) = 0.
$$
Therefore $w^* y^*yw = 0$ and hence $y = 0$ since $w$ is invertible. It immediately follows that $\Bic(M, \varphi) = \C 1$. This finishes the proof of Theorem \ref{bicentralizer}.
\end{proof}

\bibliographystyle{plain}

\end{document}